\documentclass[twocolumn]{autart}
\makeatletter
\let\c@author\relax
\makeatother
\usepackage{biblatex}
\makeatletter
\let\theoremstyle\relax
\makeatother
\usepackage{amsmath}
\usepackage{amsthm}
\usepackage{amsfonts}
\usepackage{graphicx}
\DeclareGraphicsExtensions{.pdf}
\graphicspath{{./images/}}
\newtheorem{lemma}{Lemma}
\newtheorem{theorem}{Theorem}

\newtheorem{remark}{Remark}
\newtheorem{assumption}{Assumption}

\newtheorem{proposition}{Proposition}

\usepackage{caption}
\usepackage{subcaption}
\usepackage{mathtools}
\usepackage[final]{changes}
\usepackage{algpseudocode}
\usepackage{algorithm}

\definecolor{newred}{RGB}{229, 0, 76}
\usepackage{amssymb} 
\usepackage{xcolor}
\definecolor{kthblue}{rgb}{0.098, 0.329, 0.651}
\usepackage{bbm}
\usepackage{hyperref}

\begin{document}
\begin{frontmatter}
\title{
Optimal control of
continuous-time symmetric systems\\
with unknown dynamics and noisy measurements
\thanksref{footnoteinfo}}
\thanks[footnoteinfo]{
This work was supported by KTH Royal Institute of Technology.}
\author[ethz]{Hamed Taghavian}\ead{htaghavian@ethz.ch},
\author[ethz]{Florian Dorfler}\ead{dorfler@ethz.ch},
\author[kth]{Mikael Johansson}\ead{mikaelj@kth.se}
\address[kth]{KTH Royal Institute of Technology, Stockholm, Sweden}
\address[ethz]{ETH Zurich, Zurich, Switzerland}
\normalsize{\textbf{Link to the published version: \textcolor{kthblue}{\url{https://doi.org/10.1016/j.automatica.2025.112609}}}}


\begin{abstract}
An iterative learning algorithm is presented for continuous-time linear-quadratic optimal control problems where the system is externally symmetric with unknown dynamics. Both finite-horizon and infinite-horizon problems are considered. It is shown that the proposed algorithm is globally convergent to the optimal solution and has some advantages over adaptive dynamic programming, including being unbiased under noisy measurements and having a relatively low computational burden. Numerical experiments show the effectiveness of the results.
\end{abstract}

\end{frontmatter}

\section{introduction}
    Continuous-time linear-quadratic regulation (LQR) is a classical problem in optimal control theory that has been studied for decades. This problem aims to minimize a quadratic cost subject to the dynamics of a linear system
    \begin{equation}\label{eqn:cost_intro}
    \begin{array}[c]{rll}
    \underset{u}{\text{minimize}} & \int_0^{t_f} x'(t)Q_x x(t)+u'(t)Ru(t)dt\\
    \mbox{subject to} & \frac{d}{dt}{x}(t)=Ax(t)+Bu(t), \quad t\geq 0
    \end{array}
    \end{equation}    
    where $x(t)\in\mathbb{R}^n$ is the state, $x(0)=x_0$ is the initial condition and $u(t)\in\mathbb{R}^m$ is the input. When the model $(A,B)$ is known, the optimal solution to the problem (\ref{eqn:cost_intro}) can be derived using dynamic programming as the state feedback control law
    \begin{subequations}
        \begin{align}
        u^{\star}(t)&=-K(t)x(t) \label{eqn:u=-Kx}\\
        K(t)&=R^{-1}B'P(t), \label{eqn:lqrGain_FH}
        \end{align}
    \end{subequations}
    where $P(t)$ is obtained from the associated Riccati differential equation
    \begin{align}\label{eqn:riccati_FH}
        -\frac{d}{dt}P(t)&=A'P(t)+P(t)A-P(t)BR^{-1}B'P(t)\nonumber\\
        &+Q_x, \quad P(t_f)=0
    \end{align}
    or alternatively, from a two-point boundary problem~\cite[\S3.3]{lewis}. Without access to the system model, however, solving the optimal control problem (\ref{eqn:cost_intro}) is non-trivial. One possible way to handle this problem is to first identify a model and then use it to solve the optimal control problem. This is referred to as the indirect approach. Alternatively, one can learn the optimal controller directly, by observing the input-output data while interacting with the system~\cite{dorfler2023}. This paper focuses on the latter direct approach, for which there are numerous methods available in discrete-time domain, such as policy gradient methods~\cite{faz2018}, behavioral approaches~\cite{deper2019}, and off-policy reinforcement learning~\cite{lopez2023efficient}. In contrast, the results for continuous-time problems as (\ref{eqn:cost_intro}) are much more limited. Hence, we focus on direct methods for the optimal control of continuous-time systems with unknown dynamics in this paper.

    \subsection{State-of-the-art}    
    Kleinman's algorithm sets a stepping stone towards solving the LQR problem (\ref{eqn:cost_intro}) without accessing the system model, by using a policy iteration. This algorithm starts from a state-feedback gain, improves this gain by ensuring a descent in the cost function at every iteration, and finally, converges to the optimal gain~(see \cite{kleinman_FH} for finite-horizon and \cite{kleinman_infH} for infinite-horizon cases). While this procedure still requires knowing the system model, it can be modified to solve the LQR problem in a model-free way. The early such results reformulate the policy evaluation step in Kleinman's algorithm (a cost-to-go matrix differential equation in the finite-horizon case and a Lyapunov equation in the infinite-horizon case) to only use a partial knowledge of the model. This has been demonstrated in~\cite{IET} and \cite{vrabie2009} for finite and infinite-horizon problems respectively. Later in \cite{lee2012,auto2012}, completely model-free algorithms emerged from Kleinman's algorithm that can solve the LQR problem without requiring the system model.

     \subsubsection{Infinite horizon problems}
     The model-free approaches mentioned above were further polished with improved computational complexity in~\cite{auto2016}, and eventually, a set of algorithms appeared that can solve infinite-horizon optimal control problems with unknown models based on adaptive dynamic programming~\cite{auto2022,TAC2019,trancyber2022,TAC2016}. Put simply, in these methods, an initial control input is applied to the system, and the state trajectory is recorded in a finite time interval. This recorded data is then used to update the estimated cost and gain repeatedly by solving a linear algebraic equation in each iteration $k\geq 0$ as follows
    \begin{equation}\label{ADPupdate}
        \Theta_k \kappa_k =\theta_k,
    \end{equation}
    where $\Theta_k\in\mathbb{R}^{l\times r}$, $\theta_k\in\mathbb{R}^{l}$ are constructed by the observed data and $\kappa_k \in\mathbb{R}^{r}$ includes the elements of estimated gain and cost. To ensure that (\ref{ADPupdate}) has a unique solution, the observed input-output data must be information-rich so that $\Theta_k$ has a full column rank. This requirement leads to a set of rank conditions, such as initial excitation, persistence of excitation, etc., that are needed for these algorithms to converge~\cite{auto2016,auto2012,auto2022,TAC2016,trancyber2022}.  However, it is not clear how many data points are required to satisfy these conditions. Furthermore, since the states or system outputs are measured before feeding the model-free optimal control algorithms, they can be contaminated by measurement noise, which can have a significant impact on the performance of these algorithms. Even when a white measurement noise is present, the update rule (\ref{ADPupdate}) gets biased in \cite{auto2016,auto2012,TAC2016,trancyber2022}, \emph{i.e.},
    $$
    \mathbb{E}(\tilde{\Theta}_k) \neq \Theta_k 
    ,\;
    \mathbb{E}(\tilde{\theta}_k) \neq \theta_k
    $$
    holds in general, where $\tilde{\Theta}_k$ and $\tilde{\theta}_k$ denote the corresponding data matrices in the presence of measurement noise.

    \subsubsection{Finite horizon problems}
    Model-free finite-horizon optimal control problems have drawn much less attention than their infinite-horizon counterparts in the literature. These problems can be more challenging, as their optimal solution cannot be described by a time-invariant control law in general. This necessitates an algorithm that inevitably operates on continuous-time signals (an infinite-dimensional space) rather than constant gains (a finite-dimensional space), which increases the demands on memory usage and computational burden significantly. Nevertheless, a few algorithms have been proposed to solve the finite-horizon optimal control problems when the system dynamics are completely unknown. Remarkable works in this domain are mentioned in order: A dual-loop algorithm was proposed based on Kleinman's algorithm in~\cite{dual2017} that actively improves the estimated gain (outer loop) and uses the estimated gain to make new measurements within each iteration (inner loop). In~\cite{singular}, the singular perturbation method was applied to problems that have large horizon lengths compared to the system dynamics and a Hamiltonian matrix whose eigenvalues are off the imaginary axis. Finally, in~\cite{TAC2021}, an iterative algorithm is proposed that can also be used for more general optimal control problems than (\ref{eqn:cost_intro}), \emph{e.g.}, with nonlinear system dynamics. However, this algorithm does not converge to the optimal solution of unconstrained LQR problems (\ref{eqn:cost_intro}) in general.  

    \subsection{Contributions}
    In this paper, we consider continuous-time linear-quadratic output-regulation optimal control problems with unknown dynamics. We provide an algorithm that can be used for both finite-horizon and infinite-horizon problems. We assume that the unknown system is externally symmetric. This assumption is satisfied for all single-input-single-output (SISO) systems. External symmetry is a property of the input-output relation of the system, which is milder than state-space symmetry. Hence, it does not impose any restrictions on the matrices $A$ and $B$ alone in (\ref{eqn:cost_intro}). Rather, a restriction is imposed when the system outputs are considered. External symmetry is inherently present in many applications even when the system models are unknown, such as mechanical systems~\cite{symech}, electrical circuits~\cite{symRLC2}, chemical systems~\cite{plants} and vehicle platoons~\cite{symdata}. For example, all RLCT systems \emph{i.e.}, networks comprising resistors (R), inductors (L), capacitors (C), and transformers (T) are symmetric. Recall that these elements have precise analogs in several other engineering fields as well, including mechanics, hydraulics, and thermodynamics~\cite{symRLC}.
    
    The algorithm proposed in this paper uses the input-output data measured online to solve the optimal control problem by finding the fixed point of Pontryagin's equations. Our algorithm can be understood as an iterative learning control (ILC) law. However, unlike the classical ILC, the proposed algorithm converges to the optimal solution of the LQR problem instead of a predefined reference trajectory. In line with the spirit of ILC methods, the proposed algorithm updates the control signal directly without assuming a specific feedback structure. This is beneficial with regards to memory usage and computational complexity since $m$ parameters in the control input $u(t)$ are updated rather than $mn$ parameters in the feedback gain $K(t)$. Moreover, the optimal feedback gain can be obtained afterward, by using the optimal control signal and $n$ information-rich data points obtained after the learning process. Recently, ILC algorithms have been considered for solving finite-horizon optimal control problems in the literature, by appending auxiliary optimization problems to the ILC algorithms~\cite{ILCaux}. For SISO systems, our algorithm falls back to a continuous-time counterpart of the ILC algorithm in~\cite{bingchu}.    
    
    The main features of the proposed method compared with the state-of-the-art are summarized below:
    
        \subsubsection{Convergence conditions} For finite-horizon problems, the proposed algorithm is guaranteed to converge to the optimal solution with a linear rate from any initial controller, with no restrictions on the horizon length or the associated Hamiltonian matrix eigenvalues. For infinite-horizon problems, convergence is assured as long as the system gain is small.

        \subsubsection{Measurement noise} In the presence of measurement noise, the algorithms~\cite{auto2016,auto2012,TAC2016,trancyber2022,dual2017,singular} get biased and no longer converge to the optimal solution in expectation. On the contrary, the proposed algorithm in this paper provides unbiased updates with bounded variance under noisy measurements. Therefore, the optimal solutions to both finite and infinite-horizon problems can be found by averaging the solutions obtained from noisy measurements.
        
        \subsubsection{Computational complexity} Most existing algorithms require solving a linear algebraic equation like~(\ref{ADPupdate}) at each iteration, resulting in a computational complexity of at least \( O(lr^2) \) per iteration. For instance, \( r = n^2 + nm \) in~\cite{singular}, \( r \geq \frac{n(n+1)}{2} + nm \) in~\cite{auto2012,trancyber2022,TAC2019,TAC2016,auto2016,auto2022}. In~\cite{dual2017}, the complexity is even higher, with \( l \geq N \) and \( r = N\bigl(\frac{n(n+1)}{2} + nm\bigr) \), where \( N \) is the number of intervals used to discretize the horizon \([0, t_f]\). In contrast, the proposed algorithm has a significantly lower computational complexity of \( O(Nm^2) \) per iteration.

        \subsubsection{Intelligent data storage}
        Most data-driven algorithms used for infinite-horizon problems need a so-called intelligent data storage to ensure the recorded data is informative enough for convergence~\cite{auto2016,auto2012,auto2022,TAC2016,trancyber2022}. This is often expressed as a rank condition similar to the persistence of excitation. However, verifying these conditions is non-trivial and the number of data points needed for this purpose is not known in advance. The proposed algorithm in this paper also needs informative data to find the optimal solution to infinite-horizon problems. However, this requirement is met generically (in the measure-theoretic sense) by using $n$ equally distanced data points.
        

        \subsubsection{Exploration noise} An exploration noise is commonly added to the input signals in adaptive dynamic programming, which can help to make the measured data informative and satisfy the rank conditions described above~\cite{auto2016,TAC2016,auto2012,TAC2019,auto2022}. Choosing this exploration noise is not a trivial task and is often done in a heuristic way. The presented algorithm in this paper does not require an exploration noise.

\subsection{Notation}
The following notation is used throughout the paper. For a matrix $M$, we use $[M]_{ij}$ to refer to the element on the $i^{\rm th}$ row and $j^{\rm th}$ column of $M$, $M'$ is the transpose of $M$, and $\vert M\vert$ is a matrix with components $\left[\vert M\vert\right]_{ij}=\vert m_{ij}\vert$. $I$ is used to denote the identity matrix of appropriate dimensions. The Kronecker product is denoted by $\otimes $ and ${\rm vec}:\mathbb{R}^{m\times n}\to \mathbb{R}^{mn}$ is the vectorization operator. The relation $S\succ 0$ ($S\succeq 0$) means that matrix $S$ is symmetric positive (semi)-definite and $S^{1/2}$ is the principal square root of $S\succeq 0$. $\lambda_{\rm min} S$ and $\lambda_{\rm max} S$ denote the smallest and largest eigenvalues of the symmetric matrix $S$. The square diagonal matrix $\Sigma\in\mathbb{R}^{m\times m}$ is called a signature matrix if $[\Sigma]_{ii}\in\lbrace -1,+1\rbrace$ for all $i=1,2,\dots,m$. $\Vert .\Vert_p$ denotes the standard $p$-norm for vectors and matrices ($p\in\mathbb{N}\cup \infty$), $\Vert G\Vert_{\rm pk}$ is the peak-to-peak gain and $\Vert G\Vert_{H_\infty}$ is the $H$-infinity norm of the linear system $G(s)$. We use ${\rm Cov}\left(x;t\right)$ to denote the covariance matrix at time $t$ associated with the continuous-time random process $x:[0,t_f]\to \mathbb{R}^m$. For the functionals $f,g:\mathbb{R}^n\to \mathbb{R}$, we write $f(x)=O\bigl(g(x)\bigr)$ as $x\to\infty$, if there are some positive constants $r$, $a$ such that $\vert f( x)\vert\leq r\vert g(x)\vert$ holds for all $\Vert x\Vert_{\infty}>a$. Finally, the Dirac delta and the Heaviside step functions are denoted by $\delta(t)$ and $\mathbf{1}(t)$ respectively.


\section{Preliminaries}
Consider the continuous-time linear time-invariant system
\begin{align}\label{eqn:system}
\begin{array}{rcl}
    \frac{d}{dt}{x}(t)&=&Ax(t)+Bu(t) \\
    y(t)&=&Cx(t)+Du(t) 
\end{array},&\quad t\geq 0
\end{align}
where $x(t)\in\mathbb{R}^n$, $u(t)\in\mathbb{R}^m$ and $y(t)\in\mathbb{R}^m$. The associated transfer function matrix with (\ref{eqn:system}) is denoted by $G(s)\in\mathbb{C}^{m\times m}$, where
\begin{equation}\label{eqn:G(s)}
    G(s)=C(sI-A)^{-1}B+D
\end{equation}
and $g(t)=C\exp(At)B\mathbf{1}(t)+D\delta(t)$ is the impulse response. Let $\mathcal{L}^m_p(0,t_f)$ be the space of (equivalence classes of) all vector-valued functions $u:[0,t_f]\to\mathbb{R}^m$ with Lebesgue measurable components that have bounded-$p$ norm $\Vert u\Vert_{p,t_f}$ defined by
\begin{equation}\label{signal_normtf}
    \Vert u\Vert_{p,t_f}:=\left( \int_0^{t_f}\Vert u(t)\Vert_p^p dt\right)^{1/p}
\end{equation}
for $t_f\in(0,+\infty ]$ and $p\in\mathbb{N}\cup\lbrace \infty\rbrace$. Throughout this paper, we do not distinguish functions that are equal almost everywhere. A linear operator $\mathcal{K}$ is called bounded (in the space $\mathcal{L}^m_p(0,t_f)$), if there exists some constant $\gamma(\mathcal{K})\in(0,+\infty)$ such that
$$
\Vert \mathcal{K}u\Vert_{p,t_f} \leq \gamma(\mathcal{K})\Vert u\Vert_{p,t_f}
$$
for all $u\in\mathcal{L}^m_p(0,t_f)$. The smallest such constant $\gamma(\mathcal{K})$ is called the operator norm which is denoted by $\Vert\mathcal{K}\Vert_{p,t_f}$.
The identity operator $\mathcal{I}$ and the time-reversal operator $\mathcal{J}$ defined as
\begin{equation}\label{eqn:O_timereversal}
    \mathcal{J}u(t):={\hat{u}}(t)=u(t_f-t)
\end{equation}
are both bounded linear operators with unity norms. The linear-time invariant system (\ref{eqn:system}) constitutes an affine operator as $y=\mathcal{G}u+d_{x_0}$, where
\begin{equation}\label{eqn:Gu_linearoperator}
         \mathcal{G}u(t)=\int_0^{t} g(t-\tau)u(\tau)d\tau , \quad t\in[0,t_f]  
\end{equation}
and $d_{x_0}(t)=C\exp(At)x_0$ is the natural response of (\ref{eqn:system}) to the initial condition $x(0)=x_0$. The input-output gain of the system (\ref{eqn:system}) on the finite time interval $t\in[0,t_f]$ is defined based on the linear part (\ref{eqn:Gu_linearoperator}) as follows
\begin{equation}\label{system_normtf}
    \Vert \mathcal{G}\Vert_{p,t_f}:=\sup_{u\neq 0} \frac{\Vert \mathcal{G}u\Vert_{p,t_f}}{\Vert u\Vert_{p,t_f}}.
\end{equation}
The fundamental multiple-input-multiple-output (MIMO) system norms can be recovered from definition (\ref{system_normtf}) by allowing $t_f\to \infty$ as \emph{e.g.}, the $H$-infinity gain $\Vert \mathcal{G}\Vert_{2,\infty}=\Vert G\Vert_{H_\infty}$ and the peak-to-peak gain $\Vert \mathcal{G}\Vert_{\infty,\infty}=\Vert G\Vert_{\rm pk}$ (there is no general consensus on the definition of peak-to-peak gain for MIMO systems~\cite{apkarian2022mixed}). As the following proposition shows, finite-horizon gains (\ref{system_normtf}) are monotonic with respect to the horizon length $t_f$. Therefore, the finite-horizon system gains (\ref{system_normtf}) lower-bound the conventional (infinite-horizon) system norms.
\begin{proposition}\label{prop:monogain}
   If $t_1<t_2$, then
   $$
   \Vert \mathcal{G}\Vert_{p,t_1}\leq \Vert \mathcal{G}\Vert_{p,t_2}
   $$
   for all $p\in\mathbb{N}\cup\lbrace \infty\rbrace$.
\end{proposition}
\begin{proof}
    Define
    $$\bar{\mathcal{L}}_p^m(0,t_1)=\left\lbrace u\in\mathcal{L}_p^m(0,t_2)\,\vert\, u(t)=0,t\in(t_1,t_2]\right\rbrace
    $$
    and let $y(t)=\mathcal{G}u(t)$ be given by (\ref{eqn:Gu_linearoperator}) where $t_f=t_2$. One can write
\begin{align*}
    \Vert \mathcal{G}\Vert_{p,t_1}=\sup_{u\in\mathcal{L}_p^m(0,t_1)\backslash 0} \frac{\Vert y\Vert_{p,t_1}}{\Vert u\Vert_{p,t_1}}=\sup_{u\in\bar{\mathcal{L}}_p^m(0,t_1)\backslash 0} \frac{\Vert y\Vert_{p,t_1}}{\Vert u\Vert_{p,t_2}}\leq \\
    \sup_{u\in\bar{\mathcal{L}}_p^m(0,t_1)\backslash 0} \frac{\Vert y\Vert_{p,t_2}}{\Vert u\Vert_{p,t_2}}
    \leq \sup_{u\in\mathcal{L}_p^m(0,t_2)\backslash 0} \frac{\Vert y\Vert_{p,t_2}}{\Vert u\Vert_{p,t_2}}=\Vert \mathcal{G}\Vert_{p,t_2}.
\end{align*}
\end{proof}
We only consider the space $\mathcal{L}^m_p(0,t_f)$ with $p=2$ for finite-horizon problems and $p=\infty$ for infinite-horizon problems in this paper.

\subsection{$\mathcal{L}^m_2(0,t_f)$}
The space $\mathcal{L}^m_2(0,t_f)$ is complete (Hilbert) with the inner product $\langle u,v\rangle_{t_f}=\int_{0}^{t_f}u'(t)v(t)dt$ and the induced norm $\Vert u\Vert_{2,t_f}=\sqrt{\langle u,u\rangle_{t_f}}$. The (linear part of) system (\ref{eqn:system}) is a bounded linear operator in this space, whether or not the system (\ref{eqn:system}) is stable. The boundedness of $\mathcal{G}$ can be easily shown by separating its direct-feedthrough term as $\mathcal{G}u=\mathcal{G}_{sp}u+Du$ and using the triangle inequality
\begin{align*}
\Vert \mathcal{G}u \Vert_{2,t_f} &\leq \Vert \mathcal{G}_{sp}u\Vert_{2,t_f}+ \Vert Du\Vert_{2,t_f} \\
&\leq (\gamma(\mathcal{G}_{sp}) +\Vert D\Vert_{2})\Vert u\Vert_{2,t_f},
\end{align*}
in which the constant $\gamma(\mathcal{G}_{sp})$ is given by~\cite{kouba2020}
    $$
    \gamma(\mathcal{G}_{sp})=\iint_{[0,t_f]^2}\Vert C\exp(A(t-\tau))B\Vert_2^2 dt d\tau <\infty.
    $$
For every bounded linear operator $\mathcal{K}$ in $\mathcal{L}^m_2(0,t_f)$, there exists a unique bounded linear adjoint operator $\mathcal{K}^*$ such that $\langle \mathcal{K}u,v\rangle_{t_f}= \langle u,\mathcal{K}^*v\rangle_{t_f}$
holds for all $u,v\in\mathcal{L}^m_2(0,t_f)$. The following proposition gives the adjoint of an integral operator as
\begin{equation}\label{eqn:integraloperator}
\mathcal{K}u(t)=\int_0^{t_f}K(t,\tau)u(\tau)d\tau,
\end{equation}
where the kernel components are extended real-valued functions $[K]_{ij}:[0,t_f]^2\to  \mathbb{R}\cup \lbrace{\pm \infty}\rbrace$.

\begin{proposition}[\cite{kouba2020}]\label{prop:adjoint}
    The adjoint of the bounded linear operator (\ref{eqn:integraloperator}) is given by
    $$
    \mathcal{K}^*u(t)=\int_{0}^{t_f} K'(\tau,t)u(\tau)d\tau.
    $$
\end{proposition}
A linear operator $\mathcal{K}$ is called non-negative, if $\langle \mathcal{K}u,u\rangle_{t_f} \geq 0$ holds for all $u\in\mathcal{L}^m_2(0,t_f)$. 


\subsection{$\mathcal{L}^m_{\infty}(0,t_f)$}
All the functions in $\mathcal{L}^m_{\infty}(0,t_f)$ have (essentially) bounded components. In addition, the system norm in this space admits an explicit expression in the time domain as follows
\begin{equation}\label{Ginftf}
    \left\Vert \mathcal{G}\right\Vert_{\infty,t_f}=\left\Vert \int_0^{t_f} \vert g(t)\vert dt\right\Vert_{\infty},
\end{equation}
where the absolute value is understood component-wise. Identity (\ref{Ginftf}) shows that $\left\Vert \mathcal{G}\right\Vert_{\infty,t_f}$ is bounded when $t_f$ is finite, even if system (\ref{eqn:system}) is unstable.

\subsection{Problem statement}
In this paper, we consider the output-regulation optimal control problem 
\begin{equation}\label{eqn:optimization}
    \begin{array}[c]{rll}
    \underset{u\in\mathcal{L}^m_{p}(0,t_f)}{\text{minimize}} & J(x_0,u)=\frac{1}{2}\int_{0}^{t_f} y'(t)Qy(t)+u'(t)Ru(t) dt\\
    \mbox{subject to} & x,u,y \textnormal{ satisfy (\ref{eqn:system})}\\
    &x(0)=x_0,
    \end{array}
\end{equation}
where $Q\succeq 0$, $R\succ 0$. We are interested in finding the optimal solution $u^{\star}$ to the problem (\ref{eqn:optimization}) purely based on the input-output data observed from system~(\ref{eqn:system}). The system dynamics (\ref{eqn:system}), \emph{i.e.}, the state-space matrices $A$, $B$, $C$ and $D$ are assumed unkown. However, we make the following assumption throughout the paper.

\begin{assumption}\label{ass:sym}
    System (\ref{eqn:system}) is externally symmetric.
\end{assumption}

\section{Symmetric systems}
A system (\ref{eqn:system}) is called externally symmetric if there is a signature matrix $\Sigma_e$ such that
\begin{equation}\label{eqn:G'=G}
    \Sigma_e G'(s)=G(s) \Sigma_e.
\end{equation}
Since the identity matrix is also a signature matrix, all systems with symmetric transfer functions (and, in particular, all SISO systems) are externally symmetric. This special case of external symmetry with $\Sigma_e=I$ is also referred to as ``system symmetry" in the literature~\cite{lam2019}. A System (\ref{eqn:system}) is said to be \emph{internally} symmetric if the matrix
$$
\begin{bmatrix}
    -\Sigma_i & 0 \\
    0 & \Sigma_e
\end{bmatrix}
\begin{bmatrix}
    A & B \\
    C & D
\end{bmatrix}
$$
is symmetric for some signature matrices $\Sigma_e$ and $\Sigma_i$. Matrices $\Sigma_e$ and $\Sigma_i$ are called the external and internal signature matrices of the symmetric system (\ref{eqn:system}), respectively.
It is easy to show that an internally symmetric system is also externally symmetric (with the external signature $\Sigma_e$ and any internal signature $\Sigma_i$), because
\begin{align*}
\Sigma_e G'(s)&=\Sigma_e B' (sI-A')^{-1} C' + \Sigma_e D' \\
&=\Sigma_e (-\Sigma_e C \Sigma_i) (sI-\Sigma_i A\Sigma_i)^{-1} (-\Sigma_i B \Sigma_e)\\
&+ \Sigma_e (\Sigma_e D \Sigma_e)\\
&=C \Sigma_i (sI-\Sigma_i A\Sigma_i)^{-1} \Sigma_i B \Sigma_e + D \Sigma_e\\
&=C (sI- A)^{-1} B \Sigma_e + D \Sigma_e\\
&=G(s) \Sigma_e,
\end{align*}
where we have used the fact that any signature matrix $\Sigma$ satisfies $\Sigma=\Sigma'=\Sigma^{-1}$. Conversely, an externally symmetric system always admits an internally symmetric realization with some suitable internal signature $\Sigma_i$~\cite{wil1976}. Though, obviously, not every state-space is externally symmetric. The following proposition characterizes the family of externally symmetric systems in the state space.

\begin{proposition}[\cite{wil1976}]\label{prop:exsym_ss}
Let (\ref{eqn:system}) be a minimal relaization of (\ref{eqn:G(s)}). The system transfer function (\ref{eqn:G(s)}) satisfies (\ref{eqn:G'=G}) if and only if there exists unique symmetric matrices $S$ and $T$ with $T$ being non-singular, such that
\begin{align}\label{eqn:A=ST}
    A&=ST\nonumber\\
    C &= \Sigma_e B' T\nonumber\\
    D&=\Sigma_e D' \Sigma_e.
\end{align}    
\end{proposition}

It was first shown in \cite[p3–15]{fro1910} that all square matrices can be decomposed as two symmetric matrices as $A=ST$~\cite{Uh2013}. This decomposition was rediscovered in \cite{taussky1959}, and it was later shown to be unique given the second condition in (\ref{eqn:A=ST})~\cite{wil1976}. Proposition~\ref{prop:exsym_ss} implies that external symmetry is an \emph{input-output} property and does not impose any restrictions on the internal model of the system. Any system is externally symmetric with a suitable output because, for all pairs $(A,B)$, one can define a suitable pair $(C,D)$ such that (\ref{eqn:A=ST}) holds and the system~(\ref{eqn:system}) becomes externally symmetric. Furthermore, the matrix $C$ can be absorbed in the cost function so that the optimization problem~(\ref{eqn:optimization}) boils down to a standard LQR problem as~(\ref{eqn:cost_intro}).

A special case of symmetric systems frequently encountered in applications is the family of (internally) \emph{completely} symmetric systems. These systems are symmetric with the signatures $\Sigma_i=-I$ and $\Sigma_e=I$, \emph{i.e.}, they satisfy
\begin{align}\label{eqn:int.com.sym}
    A&=A'\nonumber\\
    B&=C'\nonumber\\
    D&=D'.
\end{align}
Property (\ref{eqn:int.com.sym}) is often called ``state-space symmetry" in the literature. Obviously, a completely symmetric system has a symmetric transfer function, however, not all symmetric transfer functions admit symmetric state-space realizations as (\ref{eqn:int.com.sym})~\cite{lam2019}. A well-known special case of completely symmetric systems is the family of relaxation systems~\cite{Pat2019}, which in addition to (\ref{eqn:int.com.sym}), also satisfy $A\preceq 0$, $D\succeq 0$.

\section{Main results}
We present an algorithm that can solve the optimal control problem (\ref{eqn:optimization}) without any prior model knowledge apart from external symmetry (\ref{eqn:G'=G}). Starting with an arbitrary initial input signal, it involves repeatedly running slightly modified time-reversed output signals through the system until convergence. Consider the affine operator $\mathcal{T}\in\mathcal{L}_p^m(0,t_f)\times \mathcal{L}_p^m(0,t_f)$ as follows
\begin{equation}\label{eqn:T}
\mathcal{T}:=-R^{-1}\Sigma_e \mathcal{J}\mathcal{G}\Sigma_e Q \mathcal{J}(\mathcal{G}+d_{x_0}),  
\end{equation}
as illustrated in Figure (\ref{fig:operator}). This operator can be decomposed as $\mathcal{T}=\mathcal{S}+r_{x_0}$ where
\begin{align}  
\mathcal{S}&=-R^{-1}\Sigma_e \mathcal{J}\mathcal{G}\Sigma_e Q \mathcal{J}\mathcal{G}=\mathcal{T}_{|x_0=0},\label{eqn:S}\\
r_{x_0}&=-R^{-1}\Sigma_e \mathcal{J}\mathcal{G}\Sigma_e Q \mathcal{J}d_{x_0}.\nonumber
\end{align}
Since all the operators (including the linear system (\ref{eqn:Gu_linearoperator})) in series are bounded in (\ref{eqn:S}), $\mathcal{S}$ is a bounded linear operator in $\mathcal{L}_p^m(0,t_f)$ and $r_{x_0}\in\mathcal{L}_p^m(0,t_f)$. An attractive feature of operator $\mathcal{T}$ is that it can be evaluated online by only using the input-output data of the system. Therefore, as the next lemma shows, this operator can be used to solve Problem~(\ref{eqn:optimization}) without knowing the system model. When there is no confusion we omit the time arguments for conciseness.

\begin{figure*}
	\begin{center}
    \includegraphics[width=0.75\linewidth]{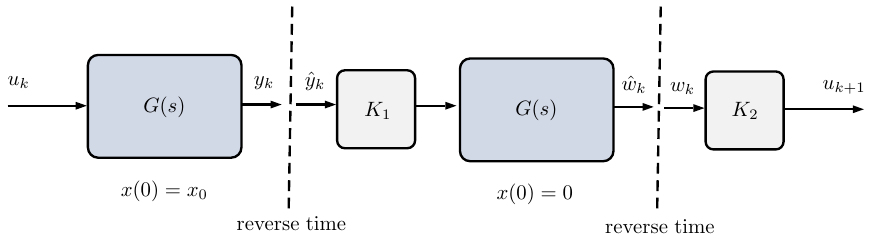}  
	\caption{Operator $u_{k+1}=\mathcal{T}(u_{k})$ defined in (\ref{eqn:T}) whose fixed point coincides with the optimal solution of problem (\ref{eqn:optimization}) by the Pontryagin's principle (\ref{eqn:pontrya}). The constant gains are given by $K_1=\Sigma_e Q$ and $K_2=-R^{-1}\Sigma_e$ and the dashed lines show the time-reversal operator (\ref{eqn:O_timereversal}).}
	\label{fig:operator}
	\end{center}
\end{figure*}

\begin{lemma}\label{lem:Tu=u}
    Let Assumption~\ref{ass:sym} hold. The control input $u\in\mathcal{L}_p^m(0,t_f)$ is optimal in Problem~(\ref{eqn:optimization}) if and only if it is a fixed point of the operator $\mathcal{T}$ defined in (\ref{eqn:T}).
\end{lemma}
\begin{proof}
We begin with adjoining the linear system dynamics (\ref{eqn:system}) to the optimization problem (\ref{eqn:optimization}) to form the Hamiltonian
\begin{align}\label{eqn:H}
    \mathcal{H}\bigl(x(t),u(t),\lambda(t)\bigr)&=\frac{1}{2}\bigl( y'(t)Qy(t) + u'(t)Ru(t)\bigr)\nonumber\\
    &+\lambda'\bigl(Ax(t)+Bu(t)\bigr)\nonumber\\
    &=\frac{1}{2}\bigl(x'C'QCx +2u'D'QCx \\
    &+ u'(R+D'QD)u\bigr)+\lambda'(Ax+Bu).\nonumber
\end{align}
By the Pontryagin's principle~\cite{pont}, the necessary conditions of optimality are given based on (\ref{eqn:H}) as follows
\begin{subequations}\label{eqn:pontrya}
\begin{align}
    &\frac{\partial}{\partial x} \mathcal{H}\bigl(x(t),u(t),\lambda(t)\bigr)=-\frac{d}{dt}{\lambda(t)},\quad \lambda(t_f)=0\label{eqn:pontrya(x)}\\
    &\frac{\partial}{\partial u} \mathcal{H}\bigl(x(t),u(t),\lambda(t)\bigr)=0\label{eqn:pontrya(u)}.
\end{align}
\end{subequations}
Since $\mathcal{H}(x,u,\lambda)$ is convex in $x,u$ for any $\lambda\in\mathbb{R}^n$, the joint optimality condition (\ref{eqn:pontrya}) is also sufficient. The first condition (\ref{eqn:pontrya(x)}) can be written as
\begin{align}\label{eqn:step1}
    -\frac{d}{dt}\lambda(t)&=\frac{\partial}{\partial x} \mathcal{H}\bigl(x(t),u(t),\lambda(t)\bigr)\nonumber\\
    &=A'\lambda(t)+ C'QCx(t) + C'QDu(t),
\end{align}
where $\lambda(t_f)=0$. Equation (\ref{eqn:step1}) may be written in reversed time $\tau=t_f-t$ as follows
\begin{align}\label{eqn:step2}
    \frac{d}{d\tau}\hat{\lambda}(\tau)&=A'\hat{\lambda}(\tau)+C'QC\hat{x}(\tau) + C'QD\hat{u}(\tau),
\end{align}
where $\hat{\lambda}(0)=0$. As the system (\ref{eqn:system}) is externally symmetric, one may invoke Proposition~\ref{prop:exsym_ss} to write
\begin{align*}
    \frac{d}{d\tau}\hat{\lambda}(\tau)&=TS\hat{\lambda}(\tau)+T B \Sigma_e QC\hat{x}(\tau) + T B \Sigma_e QD\hat{u}(\tau),\nonumber
\end{align*}
where, multiplying the both sides by $T^{-1}$ gives
\begin{align}\label{eqn:step4}
    &\frac{d}{d\tau}\bigl(T^{-1}\hat{\lambda}(\tau)\bigr)\nonumber\\
    &= S\hat{\lambda}(\tau)+B \Sigma_e QC\hat{x}(\tau) + B \Sigma_e QD\hat{u}(\tau)\nonumber\\
    &=A T^{-1}\hat{\lambda}(\tau) + B \Sigma_e Q \bigl( C\hat{x}(\tau) + D\hat{u}(\tau)\bigr),
\end{align}
where $\hat{\lambda}(0)=0$. Equation (\ref{eqn:step4}) indicates that the signal $T^{-1}\hat{\lambda}$ coincides with the state trajectory of the system (\ref{eqn:system}) in the range $[0,t_f]$ driven by the input signal $\Sigma_e Q \left( C\hat{x} + D\hat{u}\right)$ from the origin. The output signal associated with such a trajectory is given by
\begin{equation}\label{eqn:what}
    \hat{w}(\tau):=C T^{-1}\hat{\lambda}(\tau) + D \Sigma_e Q \bigl( C\hat{x}(\tau) + D\hat{u}(\tau)\bigr).
\end{equation}
Next, we write the second optimality condition (\ref{eqn:pontrya(u)}) as
\begin{align}\label{eqn:step5}
        \frac{\partial}{\partial u} \mathcal{H}\bigl(x(t),u(t),\lambda(t)\bigr)&=D'QCx(t)+(R+D'QD)u(t)\nonumber\\
        &+B'\lambda(t)=0, \quad t\in[0,t_f]\nonumber
\end{align}
which, after time reversal, takes the form
\begin{equation}\label{eqn:step6}
     D'QC\hat{x}(\tau)+(R+D'QD)\hat{u}(\tau)+B'\hat{\lambda}(\tau)=0,
\end{equation}
where $\tau=t_f-t\in[0,t_f]$. By using Proposition~\ref{prop:exsym_ss} in (\ref{eqn:step6}), one has
\begin{align}\label{eqn:step6.5}
     0&=\Sigma_e D \Sigma_e Q C\hat{x}(\tau)+(R+\Sigma_e D \Sigma_e QD)\hat{u}(\tau)\nonumber\\
     &+\Sigma_e C T^{-1}\hat{\lambda}(\tau)\nonumber\\
     &=\Sigma_e D \Sigma_e Q \bigl( C\hat{x}(\tau)+D\hat{u}(\tau)\bigr)+\Sigma_e C T^{-1}\hat{\lambda}(\tau)+ R\hat{u}(\tau)\nonumber\\
     &=\Sigma_e \hat{w}(\tau)+R\hat{u}(\tau),
\end{align}
where $\hat{w}$ is defined in (\ref{eqn:what}). Solving equation (\ref{eqn:step6.5}) for $\hat{u}$ and reversing the time $\tau$ back to $t$ yields 
\begin{equation}\label{eqn:u=}
    u(t)=-R^{-1}\Sigma_e w(t).
\end{equation}
Equation (\ref{eqn:u=}) expresses the optimal control input in terms of a system trajectory as observed from the outputs, without requiring the explicit system dynamics. The joint conditions (\ref{eqn:step4}), (\ref{eqn:what}) and (\ref{eqn:u=}) are equivalent to the necessary and sufficient conditions of optimality (\ref{eqn:pontrya}) and can be described as the fixed-point equation
$$
u=\mathcal{T}(u),
$$
in which $\mathcal{T}$ is shown as a block diagram in Figure~\ref{fig:operator}.
\end{proof}

The following algorithm solves the problem (\ref{eqn:optimization}) by finding the fixed point of (\ref{eqn:T}). The optimal solution $u^{\star}$ is obtained within a numerical accuracy specified by $\epsilon_0$. The input parameters are $t_f>0$, $Q\succeq 0$, $R\succ 0$, and the output is $u_{k+1}$. The only tunning parameter is the step size $\alpha\in(0,1]$.

\begin{algorithm}
\caption{A model-free iterative solution to the optimal control problem (\ref{eqn:optimization}):}\label{alg}
\begin{enumerate}
    \item Choose an initial control signal $u_0$ and let $k=0$.
    \item Apply $u=u_k$ to system (\ref{eqn:system}) with the initial condition $x(0)=x_0$ and observe the output $y=y_k$.
    \item Apply $u=\Sigma_e Q\hat{y}_k$ to system (\ref{eqn:system}) with the initial condition $x(0)=0$ and observe the output $y=\hat{w}_k$.
    \item Set $\mathcal{T}(u_k)=-R^{-1} \Sigma_e w_k$.
    \item Update the control input as $u_{k+1}=(1-\alpha)u_k+\alpha\mathcal{T}(u_k)$.
    \item If $\Vert u_{k+1}-u_k\Vert < \epsilon_0$ stop. Otherwise, set $k\gets k+1$ and go back to step~2.
\end{enumerate}
\end{algorithm}

To study the convergence of Algorithm~\ref{alg} we first require the following lemma.

\begin{lemma}\label{lem:nn}
    Let Assumption~\ref{ass:sym} hold. The linear operator $-R^{1/2}\mathcal{S}R^{-1/2}$ is non-negative.
\end{lemma}
\begin{proof}
From (\ref{eqn:Gu_linearoperator}) one can write
$$
\mathcal{J}\mathcal{G}u(t)=\mathcal{G}u(t_f-t)=\int_0^{t_f-t}g(t_f-t-\tau)u(\tau)d\tau
$$
for all $t\in[0,t_f]$, where invoking Proposition~\ref{prop:adjoint} yields
\begin{align*}
(\mathcal{J}\mathcal{G})^*u(t)&=\int_0^{t_f}\Sigma_e g(t_f-t-\tau) \Sigma_e u(\tau)d\tau \\
&= (\Sigma_e \mathcal{J} \mathcal{G} \Sigma_e)u(t)
\end{align*}
due to (\ref{eqn:G'=G}). Therefore
\begin{equation}\label{eqn:adjsym}
    (\mathcal{J}\mathcal{G})^*=\Sigma_e \mathcal{J} \mathcal{G} \Sigma_e,
\end{equation}
which can be used to write
    \begin{align*}
        &\langle u,-R^{1/2}\mathcal{S}R^{-1/2}u\rangle_{t_f} \\
        &=\langle u,R^{-1/2}\Sigma_e\mathcal{J}\mathcal{G}\Sigma_e Q \mathcal{J} \mathcal{G} R^{-1/2}u\rangle_{t_f} \\
        &=\langle u,R^{-1/2}(\mathcal{J}\mathcal{G})^* Q \mathcal{J} \mathcal{G} R^{-1/2}u\rangle_{t_f} \\
        &=\langle \mathcal{J}\mathcal{G} R^{-1/2} u, Q \mathcal{J} \mathcal{G} R^{-1/2}u\rangle_{t_f} \geq 0, \quad \forall u\in\mathcal{L}^m_2(0,t_f)
    \end{align*}
where the last inequality holds due to the non-negativity of $Q$ as a linear operator in $\mathcal{L}^m_2(0,t_f)$.
\end{proof}

Now we are ready to prove the convergence of Algorithm~\ref{alg}.

\begin{theorem}\label{thm}
    Let Assumption~\ref{ass:sym} hold and $u_0\in\mathcal{L}_2^m(0,t_f)$. Algorithm~\ref{alg} converges to the optimal solution of problem (\ref{eqn:optimization}) in the space $\mathcal{L}_2^m(0,t_f)$, if $\alpha\in\bigl(0, \min\lbrace \bar{\alpha},1\rbrace \bigr)$ where
    \begin{equation}\label{eqn:conv_cond}
    \bar{\alpha}=2/\bigl(\Vert R^{1/2}\mathcal{S}R^{-1/2}\Vert^2_{2,t_f}+1\bigr).
    \end{equation} 
\end{theorem}
\begin{proof}
Accoridng to Lemma~\ref{lem:Tu=u}, the optimal solution $u^{\star}$ of Problem~(\ref{eqn:optimization}) satisfies
$$
u^{\star}=(1-\alpha)u^{\star}+\alpha\mathcal{T}(u^{\star}).
$$
Hence, subtracting $u^{\star}$ from the both sides of the update step in Algortihm~\ref{alg} yields
\begin{align}\label{(1-a)I+aS}
    &u_{k+1}-u^{\star}\nonumber\\
    &=(1-\alpha)(u_k-u^{\star})+\alpha \bigl(\mathcal{T}(u_k)-\mathcal{T}(u^{\star})\bigr) \nonumber\\
    &=(1-\alpha)(u_k-u^{\star})+\alpha \bigl(\mathcal{S}(u_k)+r_{x_0}-\mathcal{S}(u^{\star})-r_{x_0}\bigr)\nonumber\\
    &=(1-\alpha)(u_k-u^{\star})+\alpha\mathcal{S}(u_k-u^{\star})\nonumber\\
    &=\bigl( (1-\alpha)\mathcal{I}+\alpha\mathcal{S}\bigr) (u_k-u^{\star})\\
    &=R^{-1/2}\left((1-\alpha)\mathcal{I}+\alpha R^{1/2}\mathcal{S} R^{-1/2}\right) R^{1/2}(u_k-u^{\star}),\nonumber
\end{align}
which implies that for all $k\in\mathbb{N}$ one has
\begin{equation}\label{eqn:u+=Mu0}
        u_k-u^{\star}=
    R^{-1/2}\mathcal{M}_\alpha^k R^{1/2}(u_0-u^{\star}),
\end{equation}
where
\begin{equation}\label{eqn:M}
    \mathcal{M}_\alpha:=(1-\alpha)\mathcal{I}+\alpha R^{1/2}\mathcal{S} R^{-1/2}.
\end{equation}
From (\ref{eqn:u+=Mu0}), it is deduced that
\begin{equation}\label{eqn:u+<|H|u}
    \Vert u_k-u^{\star} \Vert_{2,t_f} \leq
     \left\Vert \mathcal{M}_\alpha\right\Vert_{2,t_f}^k
     \Vert R^{-1/2} \Vert_{2} \Vert R^{1/2} \Vert_{2}
    \Vert u_0-u^{\star} \Vert_{2,t_f}.
\end{equation}
To estimate the norm of ${\mathcal M}_{\alpha}$, we note that for any $u\in\mathcal{L}^m_2(0,t_f)$, one has
\begin{align}\label{eqn:<><>}
    \left\Vert \mathcal{M}_\alpha u\right\Vert_{2,t_f}^2&=
    (1-\alpha)^2\langle u,u\rangle_{t_f}\nonumber\\
    &+2\alpha(1-\alpha)\left\langle u,R^{1/2}\mathcal{S} R^{-1/2} u\right\rangle_{t_f}\nonumber\\
    &+\alpha^2 \left\langle R^{1/2}\mathcal{S} R^{-1/2}u,R^{1/2}\mathcal{S} R^{-1/2}u\right\rangle_{t_f}.
\end{align}
Since $\alpha(1-\alpha)>0$ holds for $\alpha\in(0,1)$,
one has $\alpha(1-\alpha)\left\langle u,R^{1/2}\mathcal{S} R^{-1/2} u\right\rangle_{t_f} \leq 0$ from Lemma~\ref{lem:nn}. Therefore, it follows from (\ref{eqn:<><>}) that
\begin{align}\label{Mproofstep}
    &\left\Vert \mathcal{M}_\alpha u\right\Vert_{2,t_f}\leq\nonumber\\
    & \left(
    (1-\alpha)^2 
    +\alpha^2 \Vert R^{1/2}\mathcal{S} R^{-1/2} \Vert_{2,t_f}^2 \right)^{1/2}\Vert u\Vert_{2,t_f}.
\end{align}
Choosing the step size $\alpha$ in (\ref{Mproofstep}) as
$\alpha< \bar{\alpha}$ results in
\begin{equation}\label{|M|<1}
    \left\Vert \mathcal{M}_\alpha \right\Vert_{2,t_f} <1
\end{equation}
and therefore, $\lim_{k\to +\infty}\Vert u_{k}-u^{\star}\Vert_{2,t_f} =0$ holds from inequality (\ref{eqn:u+<|H|u}).
\end{proof}

\begin{remark}[Complexity]
    Both the memory usage and computational complexity of Algorithm~\ref{alg} are dominated by the storage of the two continuous-time signals $u_k$ and $y_k$. These signals can be stored by discretizing the continuous time intervals using $N$ time samples each. The computational complexity of Algorithm~\ref{alg} per iteration is then given by $O(Nm^2)$.
\end{remark}

\begin{remark}[Model-free]
Algorithm~\ref{alg} does not require the system model and is run by only observing the input-output data of the system.
\end{remark}

\begin{remark}[Convergence rate]
Theorem~\ref{thm} ensures the linear convergence rate $O\left(\Vert\mathcal{M}_\alpha\Vert_{2,t_f}^k\right)$ in Algorithm~\ref{alg} as $k\to\infty$.
\end{remark}

\begin{remark}[Global convergence]
Algorithm~\ref{alg} always converges by choosing a small enough step size $\alpha$, regardless of the system dynamics (\ref{eqn:system}), the cost function (\ref{eqn:optimization}) or the initial control signal $u_0\in\mathcal{L}^m_{2}(0,t_f)$.
\end{remark}

\begin{remark}[Step size]
The optimal step size in Algorithm~\ref{alg} is given by $\alpha^{\star}= \bar{\alpha}/2$, computing which requires $\mathcal{S}$ in (\ref{eqn:conv_cond}), and hence, the system model. When only the system norm is known, it is possible to obtain a safe step size as $\alpha \in \bigl(0,\min\lbrace 1,\beta\rbrace\bigr)$, where
$$
\beta=2/\bigl(1+\Vert G \Vert^4_{H_\infty} \lambda^2_{\max} Q/\lambda^2_{\min}R\bigr).
$$
\end{remark}

\subsection{Measurement noise}\label{sec:FH_noise}
In practice, the measured signals can be contaminated by noise, as operator $\mathcal{T}$ in Algorithm~\ref{alg} is evaluated online by running experiments on the system. To study how this affects Algorithm~\ref{alg}, consider the system (\ref{eqn:system}) again, with a modified output equation as follows
\begin{equation}\label{eqn:noisy}
    y(t)=Cx(t)+Du(t)+\eta(t),
\end{equation}
where $\eta(t)\in\mathbb{R}^m$ is a random noise. We do not require the components of $\eta (t)$ to be independent of each other or that $\eta(t)$ be identically distributed for $t\in[0,t_f]$. However, we make the following assumption on the measurement noise $\eta$ in this section.

\begin{assumption}\label{ass:noise_FH}
    The stochastic process $\eta$ is square-integrable almost surely, \emph{i.e.}, $\mathbb{P}\bigl(\eta\in\mathcal{L}_2^m(0,t_f)\bigr)=1$, and it has a zero mean $\mathbb{E}\bigl(\eta(t)\bigr)=0$ and a bounded variance $\Vert {\rm Cov}(\eta;t)\Vert_2 \leq\sigma^2<\infty$ for $t\in[0,t_f]$.
\end{assumption}

As a consequence of the noisy outputs (\ref{eqn:noisy}), operator $\mathcal{T}$ in the $k$-th iteration of Algorithm~\ref{alg} is replaced by its stochastic version
\begin{equation}\label{eqn:Ttilde}
    \tilde{\mathcal{T}}_k:=-R^{-1}\Sigma_e \mathcal{J}(\mathcal{G}+\eta_{2,k})\Sigma_e Q \mathcal{J}(\mathcal{G}+d_{x_0}+\eta_{1,k}),
\end{equation} 
where $\eta_{1,k}$ and $\eta_{2,k}$ are independent measurement noises that appear in the first and second experiments with the system (\ref{eqn:system}) in iteration $k$ (see Figure~\ref{fig:operator}). Operator (\ref{eqn:Ttilde}) can be decomposed into the deterministic affine operator $\mathcal{T}$ (\ref{eqn:T}) and a stochastic signal $\zeta_k$ as follows 
$$
\tilde{\mathcal{T}}_k=\mathcal{T}+\zeta_k,
$$
where
\begin{align}\label{zeta}
    \zeta_k&:=-R^{-1}\Sigma_e \mathcal{J}\mathcal{G}\Sigma_e Q \mathcal{J} \eta_{1,k}-R^{-1}\Sigma_e \mathcal{J}\eta_{2,k}.
\end{align}
Under noisy measurements, Algorithm~\ref{alg} returns
\begin{equation}\label{eqn:estima}
    \tilde{u}_k(t):=u_k(t)+\nu_k(t)
\end{equation}
in iteration $k$, where $u_k(t)$ is what the algorithm returns in the nominal case in the absence of measurement noise, and $\nu_k(t)$ is a random variable. The next theorem shows that $\nu_k(t)$ has zero mean and a bounded variance. 

\begin{theorem}\label{thm:noiseFH}
    Let Assumptions~\ref{ass:sym} and \ref{ass:noise_FH} hold, let $\alpha\in\bigl(0, \min\lbrace \bar{\alpha},1\rbrace \bigr)$ where $\bar{\alpha}$ is given by (\ref{eqn:conv_cond}), and assume a deterministic initialization in Algorithm~\ref{alg} ($\nu_0=0$). Then estimation (\ref{eqn:estima}) is unbiased, \emph{i.e.},
    \begin{equation}\label{E=0FH}
        \mathbb{E}\bigl(\nu_k\bigr)=0, \quad k\in\mathbb{N}
    \end{equation}
    and is of bounded variance.
\end{theorem}
\begin{proof}
    see appendix~\ref{sec:proof_thm:noiseFH}.
\end{proof}

\subsection{Random initialization}\label{sec:random_ini}
Algorithm~\ref{alg} requires resetting the system to the same initial condition $x(0)=x_0$ at every iteration. This requirement, which is also present in most iterative-learning-control schemes, can pose some challenges in applications where the initial state of the system cannot be easily set or estimated. A possible approach to this problem is to use a random initial condition within a bounded neighborhood of $x_0$ for each iteration, run Algorithm~\ref{alg} for $\overline{\omega}$ independent trials, and average the solutions. In particular, assume the initial condition in the second step of the $k$-th iteration of Algorithm~1 (the first experiment in Figure~\ref{fig:operator}) is perturbed as
\begin{equation}\label{pert}
    x(0)=x_0+\rho_k,
\end{equation}
where $\rho_k$ is a random variable with the following assumption.

\begin{assumption}\label{ass:inicond}
    The random variables $\rho_k$ are zero-mean, independent, and identically distributed, and there is some $s>0$ such that $\Vert\rho_k\Vert_{\infty}\leq s$ holds almost surely for all $k$.
\end{assumption}

Under the perturbed initial condition (\ref{pert}), the operator $\mathcal{T}$ in the $k$-th iteration of Algorithm~1 is replaced by $\mathcal{T}+r_{\rho_k}$, where
$$
r_{\rho_k}=-R^{-1}\Sigma_e \mathcal{J}\mathcal{G}\Sigma_e Q \mathcal{J}d_{\rho_k},
$$
and $d_{\rho_k}(t)=C\exp(At)\rho_k$ is the natural response of the system (4) to the initial condition $x(0)=\rho_k$. In this case, Algorithm~1 returns
\begin{equation}\label{eqn:estimaini}
    \tilde{u}_k:=u_k+\nu_k
\end{equation}
in iteration $k$, where $u_k$ is what the algorithm returns in the nominal case (where the initial condition is set correctly $x(0)=x_0$ in all iterations), and $\nu_k$ is a random variable. The next theorem shows that $\nu_k(t)$ has a zero mean and a bounded variance. Therefore, one can average $\tilde{u}_k$ obtained from $\overline{\omega}$ independent trials to recover $u_k$. 

\begin{theorem}\label{thm:noiseini}
    Let Assumptions~\ref{ass:sym} and \ref{ass:inicond} hold, let $\alpha\in\bigl(0, \min\lbrace \bar{\alpha},1\rbrace \bigr)$ where $\bar{\alpha}$ is given by (\ref{eqn:conv_cond}), and assume a deterministic initialization in Algorithm~\ref{alg} ($\nu_0=0$). Then estimation (\ref{eqn:estimaini}) is unbiased, \emph{i.e.},
    \begin{equation}\label{E=0FHini}
        \mathbb{E}\bigl(\nu_k\bigr)=0, \quad k\in\mathbb{N}
    \end{equation}
    and is of bounded variance.
\end{theorem}
\begin{proof}
    see appendix~\ref{sec:proof_thm:noiseini}.
\end{proof}

\section{Extension to infinite-horizon problems and state feedback}\label{sec:infH}
In this section, we use the developed technique to solve the problem (\ref{eqn:optimization}) when its horizon length is infinite and derive the optimal state feedback gain. By letting $D=0$ and $Q_x=C'QC$, the output regulation problem~(\ref{eqn:optimization}) is set in the standard form of an LQR problem~(\ref{eqn:cost_intro}). As we focus on infinite-horizon problems in this section, we also make the following assumption which ensures a stabilizing control law exists and the optimal value is finite in (\ref{eqn:cost_intro}).
\begin{assumption}\label{ass:infH}
    The pair $(A,B)$ is stabilizable and $(A,\sqrt{Q_x})$ is detectable.
\end{assumption}
It is well known that when $t_f\to\infty$ and Assumption~\ref{ass:infH} is satisfied, the solution to (\ref{eqn:cost_intro}) is given by~\cite[\S3.4]{lewis}
        \begin{align}
        0&=A'P_{\infty}+P_{\infty}A-P_{\infty}BR^{-1}B'P_{\infty}+Q_x,\label{eqn:ARE}\\
        K_{\infty}&=R^{-1}B'P_{\infty}, \label{eqn:lqrGain_IH}\\
        u_{\infty}^{\star}(t)&=-K_{\infty}x(t), \label{eqn:u=-Kx_IH}
        \end{align}
where $P_{\infty}$ is the unique positive definite solution to (\ref{eqn:ARE}). Solving the above equations requires knowing the system model. In contrast, Algorithm~\ref{alg} can be used to obtain the optimal gain $K_{\infty}$ in (\ref{eqn:u=-Kx_IH}) in a model-free fashion. This is realized by first using Algorithm~\ref{alg} to obtain the optimal open-loop control signal. Then this input signal is applied to the system and a few data points are collected from the optimal state trajectory. These data points are then used to obtain the optimal state feedback gain $K_{\infty}$. For the measured data points to be reliable for this process, we need a uniform convergence in Algorithm~\ref{alg}. Hence, $\infty$-norm is used instead of $2$-norm in this section.

\begin{lemma}\label{lem:convergence_in_Linf}
    Let Assumption~\ref{ass:sym} hold and $u_0\in\mathcal{L}^m_{\infty}(0,t_f)$. Then $\lim_{k\to+\infty}\Vert u_k-u^{\star}\Vert_{\infty,t_f}=0$ holds in Algorithm~\ref{alg} for all $\alpha\in(0,1]$, if
    \begin{equation}\label{eqn:conv_cond_inf}
       \Vert  Q \Vert_\infty \Vert G\Vert_{\rm pk}^2 \Vert R ^{-1}\Vert_{\infty}< 1.
    \end{equation}
\end{lemma}
\begin{proof}
From (\ref{(1-a)I+aS}) one can write
\begin{align}\label{eqn:infproof1}
&\Vert u_{k+1}-u^{\star}\Vert_{\infty,t_f}\nonumber\\
&\leq \Vert (1-\alpha)\mathcal{I}+\alpha\mathcal{S}\Vert_{\infty,t_f} \Vert u_k-u^{\star}\Vert_{\infty,t_f} \nonumber\\
&\leq \left(1-\alpha+\alpha\Vert \mathcal{S}\Vert_{\infty,t_f}\right)
\Vert u_k-u^{\star}\Vert_{\infty,t_f}\\
&\leq \left(1-\alpha+\alpha \Vert R^{-1} \Vert_{\infty} \Vert G\Vert_{\rm pk}^2 \Vert Q\Vert_{\infty}\right)
\Vert u_k-u^{\star}\Vert_{\infty,t_f},\nonumber
\end{align}
where we have used $\Vert R^{-1}\Sigma_e\Vert_{\infty} = \Vert R^{-1}\Vert_{\infty}$, $\Vert \Sigma_e Q\Vert_{\infty} = \Vert Q\Vert_{\infty}$ and $\Vert\mathcal{J}\Vert_{\infty,t_f}=1$, and the inequality
\begin{equation}\label{Gtf<Gpk}
\Vert \mathcal{G}\Vert_{\infty,t_f} \leq 
\Vert \mathcal{G}\Vert_{\infty,\infty} =
\Vert G\Vert_{\rm pk}    
\end{equation}
that holds from Proposition~\ref{prop:monogain}. Therefore, condition (\ref{eqn:conv_cond_inf}) ensures that
\begin{align*}
&\lim_{k\to+\infty}\Vert u_k-u^{\star}\Vert_{\infty,t_f}\leq\\
&\lim_{k\to+\infty}\left(1-\alpha+\alpha \Vert R^{-1} \Vert_{\infty} \Vert {G}\Vert_{\rm pk}^2 \Vert Q\Vert_{\infty}\right)^k
\Vert u_0-u^{\star}\Vert_{\infty,t_f}\\
&=0.
\end{align*}
\end{proof}

\begin{remark}
The condition (\ref{eqn:conv_cond_inf}) ensures the system is BIBO stable and all the signals shown in Figure~\ref{fig:operator} remain bounded during the operation of Algorithm~\ref{alg} as $t_f$ grows large.
\end{remark}

Once Algorithm~\ref{alg} has converged, the optimal open-loop control signal can be used to obtain the optimal state feedback gain without requiring the system model as shown in the following theorem.

\begin{theorem}\label{thm:infHe}
Let Assumptions~\ref{ass:sym} and \ref{ass:infH} hold, assume (\ref{eqn:conv_cond_inf}) and  $u_0\in\mathcal{L}^m_{\infty}(0,t_f)$. Let $u_k$ be the control signal obtained from Algorithm~\ref{alg} after $k$ iterations with the horizon length $t_f$ and let $x_k$ be the state trajectory of the system (\ref{eqn:system}) under the control input $u_k$. Then the optimal infinite-horizon state-feedback gain $K_{\infty}$ satisfies
\begin{align}\label{eqn:Xvec(K)=-u}
    \left(X_{t_1:t_n}'\otimes I_m\right)\;{\rm vec}(K_{\infty})&= -{\rm vec}(U_{t_1:t_n}) \nonumber\\
    &+ \varepsilon(k,t_f,t_1:t_n),
\end{align}
where
\begin{align}\label{eqn:X,U_t1:tn}
    X_{t_1:t_n}&=\begin{bmatrix}
x_k(t_1)& x_k(t_2)&\dots&x_k(t_{n})   
\end{bmatrix},\nonumber\\
U_{t_1:t_n}&=\begin{bmatrix}
u_k(t_1)& u_k(t_2)&\dots&u_k(t_{n})
\end{bmatrix}
\end{align}
are data matrices with $n$ (arbitrary) samples in the range $t_i\in[0,\bar{t}]$ and
    \begin{equation}\label{eqn:eps_decay}
\Vert\varepsilon(k,t_f,t_1:t_n)\Vert_{\infty}=O\bigl(\exp(-l_1k)+\exp(-l_2t_f)\bigr),
    \end{equation}
    as $\begin{bmatrix}k & t_f\end{bmatrix}'\to \infty$, where $l_1$, $l_2$ are two positive constants given by (\ref{l1}) and (\ref{l2}) respectively.
\end{theorem}
\begin{proof}
The finite-horizon Riccati equation (\ref{eqn:riccati_FH}) can be written in reversed time as follows
\begin{equation}\label{eqn:riccati_FH_hat}
    \frac{d}{d\tau}\hat{P}(\tau)=A'\hat{P}(\tau)+\hat{P}(\tau)A-\hat{P}(\tau)BR^{-1}B'\hat{P}(\tau)+Q_x,
\end{equation}
where $\tau=t_f-t$ is the mirrored time. Under Assumption~\ref{ass:infH}, there holds $\lim_{\tau\to +\infty} \hat{P}(\tau)=P_{\infty}$ where $P_{\infty}$ is the unique positive definite solution of the algebraic Riccati equation (\ref{eqn:ARE})~\cite[Theorem~3.4.1]{lewis}. Consequently, the optimal gain (\ref{eqn:lqrGain_FH}) of the finite-horizon problem also converges to the optimal gain $K_{\infty}$ (\ref{eqn:lqrGain_IH}) of the infinite horizon problem, \emph{i.e.}, $\lim_{\tau\to +\infty} \hat{K}(\tau)=K_{\infty}$.
Define
\begin{align}\label{eqn:e^p,e^u}
    \hat{e}^P(\tau)&:=\hat{P}(\tau)-P_{\infty},\\
    e^u_k(t)&:=u_k(t)-u^{\star}(t),\nonumber
\end{align}
and let $\mathcal{G}_x$ be the linear operator mapping the system input to its state trajectory starting from the origin. It follows that
\begin{align}\label{precistep}
    -K(t)x_k(t)&=-K(t)x(t)-K(t)\bigl(x_k(t)-x(t) \bigr) \nonumber\\
    &=-K(t)x(t)-K(t)\mathcal{G}_x e^u_k(t) \nonumber\\
    &=u^{\star}(t)-K(t)\mathcal{G}_x e^u_k(t) \nonumber\\
    &=u_k(t)-e^u_k(t)-K(t)\mathcal{G}_x e^u_k(t).
\end{align}
Since the optimal gain satisfies
$$
K(t)=K_{\infty}+R^{-1}B'\hat{e}^P(t_f-t),
$$
we obtain
\begin{equation}\label{-Kx=utbar}
-K_{\infty}x_k(t)= u_k(t)+\epsilon(k,t_f,t), \quad t\in[0,t_f]
\end{equation}
from (\ref{precistep}), where $\epsilon(k,t_f,t)\in\mathbb{R}^m$ is given by
\begin{align}\label{vareps}
    \epsilon(k,t_f,t)&:=-e^u_k(t)-K(t)\mathcal{G}_x e^u_k(t)\nonumber\\
    &+R^{-1}B'\hat{e}^P(t_f-t)x_k(t).
\end{align}
By applying $u_k$ to the system (\ref{eqn:system}) and sampling $n$ points on the state trajectory in the range $t_i\in[0,\bar{t}]$, we can write
\begin{equation}\label{unvec}
- K_{\infty} X_{t_1:t_n} = U_{t_1:t_n}+E_{t_1:t_n},    
\end{equation}
where
$$
E_{t_1:t_n}=\begin{bmatrix}\epsilon(k,t_f,t_1)&\epsilon(k,t_f,t_2)&\dots&\epsilon(k,t_f,t_n)\end{bmatrix}.
$$
Vectorizing equation (\ref{unvec}) gives (\ref{eqn:Xvec(K)=-u}), where $\varepsilon(k,t_f,t_1:t_n):=-{\rm vec}(E_{t_1:t_n})$. Note that
\begin{align}\label{||eps||}
    &\Vert\varepsilon(k,t_f,t_1:t_n)\Vert_{\infty}\nonumber\\
    &= \max_{1\leq i\leq n}\Vert \epsilon(k,t_f,t_i)\Vert_{\infty} \nonumber\\
    &\leq \max_{t\in[0,\bar{t}]}\Vert \epsilon(k,t_f,t)\Vert_{\infty} \nonumber\\
    &\leq r_1 \Vert e^u_k \Vert_{\infty,t_f}+r_2 \max_{t\in[0,\bar{t}]}\Vert \hat{e}^P(t_f-t) \Vert_{\infty}
\end{align}
holds from (\ref{vareps}), where
\begin{align*}
r_1&=\Vert \mathcal{G}_x \Vert_{\infty,\bar{t}} \max_{t\in[0,\bar{t}]}\Vert K(t) \Vert_{\infty} +1,\\
r_2&=\Vert R^{-1} B' \Vert_{\infty}\Vert x_k \Vert_{\infty,\bar{t}}.
\end{align*}
Since $r_1$ and $r_2$ are bounded, it is sufficient to bound the error norms in (\ref{||eps||}) to prove (\ref{eqn:eps_decay}). We begin with $$\hat{e}^P(t_f-t)=\hat{P}(\tau)-P_{\infty}=P(t)-P_{\infty},$$ as the difference between the solutions of Riccati differential and algebraic equations (\ref{eqn:riccati_FH_hat}) and (\ref{eqn:ARE}). The analytic solution of the Riccati differential equation~(\ref{eqn:riccati_FH}) can be written as follows~\cite[\S3.4]{lewis}
\begin{align}\label{RiccatiW}
    P(t)&=\bigl( W_{21}+W_{22}V(t_f-t)\bigr)\bigl( W_{11}+W_{12}V(t_f-t)\bigr)^{-1},
\end{align}
where
\begin{equation}\label{V}
V(t_f-t)=-\exp\bigl( -\Lambda(t_f-t)\bigr) W_{22}^{-1}W_{21}\exp\bigl( -\Lambda(t_f-t)\bigr).    
\end{equation}
In (\ref{V}), $\Lambda$ is a diagonal matrix consisting of the eigenvalues of the Hamiltonian matrix
    \begin{equation}\label{Hmatrix}
        H=\begin{bmatrix}
            A & -BR^{-1}B' \\
            -Q_x & -A'
        \end{bmatrix}
    \end{equation}
in the right-half plane, and
$$
W=\begin{bmatrix}W_{11}&W_{12}\\W_{21}&W_{22}\end{bmatrix}\in\mathbb{R}^{2n\times 2n} \quad (W_{ij}\in\mathbb{R}^{n\times n})
$$
is a matrix consisting of the eigenvectors of $H$ sorted such that
$$
W^{-1}H W=\begin{bmatrix}
   -\Lambda& 0\\0&\Lambda  
\end{bmatrix}.
$$
Note that Assumption~\ref{ass:infH} guarantees that $H$ does not have any eigenvalues on the imaginary axis (this is shown in \cite{hamiltonLQR} for $R=I$, but can also be shown to be true in the general case $R\succ 0$ by input transformation). Therefore, matrix $-\Lambda$ is Hurwitz stable and $\lim_{t_f\to+\infty}V(t_f-t)=0$ holds in (\ref{V}). This results in $P_{\infty}=W_{21}W_{11}^{-1}$ from (\ref{RiccatiW}), and therefore,
\begin{align}\label{thmproof_ep}
    &\hat{e}^P(t_f-t)=\nonumber\\
    &\bigl(W_{21}+W_{22}V(t_f-t)\bigr)\bigl(W_{11}+W_{12}V(t_f-t)\bigr)^{-1}\nonumber\\
    &-W_{21}W_{11}^{-1}.
\end{align}
Since $\lim_{t_f\to+\infty}V(t_f-t)=0$, there is some sufficiently large $\bar{t}_f$ such that $\Vert W_{11}^{-1}W_{12}V(t_f-t)\Vert_{\infty}<1$ holds for all $t_f>\bar{t}_f$ and $t\in[0,\bar{t}]$. Hence, it is deduced from (\ref{thmproof_ep}) that
\begin{align*}
    &\hat{e}^P(t_f-t)=\nonumber\\
    &\bigl(W_{21}+W_{22}V(t_f-t)\bigr)\Bigl(I-W_{11}^{-1}W_{12}V(t_f-t)\nonumber\\
    &+\bigl(W_{11}^{-1}W_{12}V(t_f-t)\bigr)^2 +\dots \Bigr) W_{11}^{-1}-W_{21}W_{11}^{-1},
\end{align*}
for all $t_f>\bar{t}_f, t\in[0,\bar{t}]$, from which we obtain the following asymptotic relation
\begin{align*}
    &\Vert \hat{e}^P(t_f-t)\Vert_{\infty}\leq \\
    &\Vert V(t_f-t)\Vert_{\infty}
    \bigl(\Vert W_{21} \Vert_{\infty}\Vert W_{12} \Vert_{\infty}\Vert W_{11}^{-1} \Vert^2_{\infty} \\
    &+\Vert W_{22} \Vert_{\infty}\Vert W_{11}^{-1} \Vert_{\infty}\bigr)+O\left(\Vert V(t_f-t)\Vert_{\infty}^2\right)\nonumber\\
    &=O\left(\Vert V(t_f-t)\Vert_{\infty}\right),\quad t_f\to+\infty
\end{align*}
for any $t\in[0,\bar{t}]$. Therefore,
\begin{equation}\label{asym_p}
    \max_{t\in[0,\bar{t}]}\Vert \hat{e}^P(t_f-t)\Vert_{\infty}=O\bigl(\exp(-l_2 t_f)\bigr)
\end{equation}
holds as $t_f\to+\infty$, where
\begin{equation}\label{l2}
    l_2=2\min_{1\leq i \leq n}\operatorname{Re}\lbrace [\Lambda]_{ii} \rbrace,
\end{equation}
and $[\Lambda]_{ii}$'s are the $n$ eigenvalues of the Hamiltonian matrix (\ref{Hmatrix}) in right-half plane. 
On the other hand, it follows from (\ref{eqn:infproof1}) that
\begin{align}\label{asym_u}
    \Vert e^u_k \Vert_{\infty,t_f}&\leq\bigl(1-\alpha+\alpha\Vert R^{-1}\Vert_{\infty} \Vert G\Vert_{\rm pk}^2\Vert Q\Vert_{\infty}\bigr)^k\Vert e^u_0\Vert_{\infty,t_f} \nonumber\\
    &= O\bigl(\exp(-l_1 k)\bigr),\quad k\to+\infty
\end{align}
where
\begin{equation}\label{l1}
    l_1=-\log \bigl( 1-\alpha+\alpha\Vert R^{-1}\Vert_{\infty} \Vert{G}\Vert_{\rm pk}^2\Vert Q\Vert_{\infty} \bigr).
\end{equation}
The proof is complete by using (\ref{asym_p}) and (\ref{asym_u}) in (\ref{||eps||}).
\end{proof}

Theorem~\ref{thm:infHe} indicates that the remainder term $\varepsilon(k,t_f,t_1:t_n)$ vanishes as $k,t_f\to +\infty$ with an exponential decay rate in both $k$ and $t_f$. Therefore, the optimal infinite-horizon gain $K_{\infty}$ can be approximated by $K_{k,t_f}$, where
\begin{equation}\label{eqn:K_k,tf}
        \left(X_{t_1:t_n}'\otimes I_m\right)\;{\rm vec}(K_{k,t_f})= -{\rm vec}(U_{t_1:t_n}).
\end{equation}
Hence, the optimal gain is found by increasing $k$ (the number of iterations that Algorithm~\ref{alg} is run for) and increasing $t_f$ (the horizon length) as
$$
\lim_{k,t_f\to +\infty}K_{k,t_f}=K_{\infty}.
$$

\begin{remark}\label{remrank}
    Equation (\ref{eqn:K_k,tf}) has a unique solution if and only if $X_{t_1:t_n}$ is full-rank.
\end{remark}
The above rank condition is similar to the one used for solving discrete-time optimal control problems~\cite{depresis}. Assuming $k,t_f\to \infty$ and choosing equidistant time samples as $t_i=(i-1)\bar{t}/n$, matrix $X_{t_1:t_n}$ can be written as
$$
X_{t_1:t_n}=
\begin{bmatrix}
     x_0 & \exp(A_{\rm cl}\bar{t}/n)x_0 & \dots & \exp\left((n-1)A_{\rm cl}\bar{t}/n\right)x_0
\end{bmatrix},
$$
where $A_{\rm cl}=A-BK_{\infty}$. Hence, $X_{t_1:t_n}$ is full-rank if and only if the pair $\bigl(\exp(A_{\rm cl}\bar{t}/n),x_0\bigr)$ is controllable, which is true if $x_0$ is not orthogonal to any of the left eigenvectors of $\exp(A_{\rm cl}\bar{t}/n)$. In this case, the matrix $X_{t_1:t_n}$ is full rank generically. Therefore, with a random initial condition (drawn from any continuous probability distribution), equation (\ref{eqn:K_k,tf}) has a unique solution almost surely.


\subsection{Measurement noise}
To be able to implement the obtained state-feedback control law, we assume the system states are available through estimation or measurement, which may be contaminated by a zero-mean additive noise $\eta_x(t)\in\mathbb{R}^m$ as follows
$$
    \tilde{x}(t)=x(t)+\eta_x(t).
$$
For the measurements (\ref{eqn:noisy}) made during the learning process in Algorithm~\ref{alg}, we make the following assumption.


\begin{assumption}\label{ass:noise_IH}
    The stochastic process $\eta$ has a zero mean, \emph{i.e.}, $\mathbb{E}\bigl(\eta(t)\bigr)=0$ for $t\in[0,t_f]$, and there is some $e>0$ such that $\Vert\eta(t)\Vert_{\infty}\leq e $ holds almost surely.
\end{assumption}

Assumption~\ref{ass:noise_IH} entails that $\eta(t)$ also has a bounded variance. Similar to the finite horizon case with square-integrable noises considered in Section~\ref{sec:FH_noise}, the next theorem proves that estimation (\ref{eqn:estima}) is also unbiased with bounded variance when $t_f\to+\infty$, provided that the measurement noise is bounded.

\begin{theorem}\label{thm:noise_inf}
Let Assumptions~\ref{ass:sym}, \ref{ass:infH} and \ref{ass:noise_IH} hold. Assume~(\ref{eqn:conv_cond_inf}) and that $u_0\in\mathcal{L}^m_{\infty}(0,t_f)$ is deterministic. Estimation (\ref{eqn:estima}) is unbiased, \emph{i.e.},
    \begin{equation}\label{eqn:E(nu)=0_inf}
        \mathbb{E}\bigl(\nu_k\bigr)=0, \quad k\in\mathbb{N}
    \end{equation}
    and it is of bounded variance as $t_f\to\infty$.
\end{theorem}
\begin{proof}
    see appendix~\ref{sec:proof_thm:noise_inf}.
\end{proof}

Now assume that Algorithm~\ref{alg} is run for $k$ iterations in the presence of a measurement noise satisfying Assumption~\ref{ass:noise_IH}. Applying the control input (\ref{eqn:estima}) obtained from Algorithm~\ref{alg} results in
\begin{equation}\label{exeta}
    \tilde{x}_k(t)=x_k(t)+\mathcal{G}_x\nu_k(t)+\eta_x(t), \quad t\in[0,t_f].
\end{equation}
The optimal feedback gain can then be estimated by solving
\begin{equation}\label{eqn:Xtild.Ktild.Utild}
            \left(\tilde{X}_{t_1:t_n}'\otimes I_m\right)\;{\rm vec}(\tilde{K}_{k,t_f})= -{\rm vec}(\tilde{U}_{t_1:t_n}),
\end{equation}
where the matrices
\begin{align}\label{eqn:Xtild,Utild}
\tilde{X}_{t_1:t_n}&=
\begin{bmatrix}
\tilde{x}_k(t_1)& \tilde{x}_k(t_2)&\dots&\tilde{x}_k(t_{n})   
\end{bmatrix},\nonumber\\
\tilde{U}_{t_1:t_n}&=
\begin{bmatrix}
\tilde{u}_k(t_1)& \tilde{u}_k(t_2)&\dots&\tilde{u}_k(t_{n})   
\end{bmatrix}
\end{align}
are set up by the noisy data sampled from the system (\ref{eqn:system}). Since $\mathbb{E}\bigl(\nu_k(t)\bigr)=0$ holds from Theorem~\ref{thm:noise_inf} and $\mathbb{E}\bigl(\mathcal{G}_x\nu_k(t)\bigr)=\mathbb{E}\bigl(\eta_x(t)\bigr)=0$ holds in (\ref{exeta}) for all $t\in[0,\bar{t}]$, the matrices (\ref{eqn:Xtild,Utild}) are unbiased estimates of their deterministic counterparts in (\ref{eqn:X,U_t1:tn}), \emph{i.e.},
$$
\mathbb{E}(\tilde{X}_{t_1:t_n})=X_{t_1:t_n},\;
\mathbb{E}(\tilde{U}_{t_1:t_n})=U_{t_1:t_n}.
$$
Therefore, by averaging the data matrices obatined from $ \overline{\omega}\in\mathbb{N}$ independent trials as
$$
\overline{X}_{t_1:t_n;\overline{\omega}}=\frac{1}{\overline{\omega}}\sum_{\omega=1}^{\overline{\omega}}\tilde{X}_{t_1:t_n;\omega},
\quad
\overline{U}_{t_1:t_n;\overline{\omega}}=\frac{1}{\overline{\omega}}\sum_{\omega=1}^{\overline{\omega}}\tilde{U}_{t_1:t_n;\omega},
$$
where $\tilde{X}_{t_1:t_n;\omega}$ and $\tilde{U}_{t_1:t_n;\omega}$ denote the matrices (\ref{eqn:Xtild,Utild}) obtained in the $\omega$-th trial, one can find the optimal state-feedback gain as $\lim_{k,t_f,\overline{\omega}\to\infty}\overline{K}_{k,t_f,\overline{\omega}}=K_{\infty}$, where
$$
            \left(\overline{X}_{t_1:t_n;\overline{\omega}}'\otimes I_m\right)\;{\rm vec}\left(\overline{K}_{k,t_f,\overline{\omega}}\right)=
            -{\rm vec}\left(\overline{U}_{t_1:t_n;\overline{\omega}}\right).
$$

\section{Numerical experiments}
The proposed approach is evaluated by numerical examples in this section. The first example considers a finite-horizon optimal control problem, while the second example considers an infinite-horizon problem with noisy measurements.
\subsection{Example 1.}\label{ex:tank}
    Consider the multi-tank system in \cite{multitank2020} which is composed of 6 identical interconnected tanks. Desired is filling up the tanks to specified (possibly different) levels $h^{d}\in\mathbb{R}^6$ maintained by the steady-state inflows $v^{d}\in\mathbb{R}^6$. Linearizing this system is known to give an internally completely symmetric system, \emph{i.e.}, $A=A'$ and $B=I_6$. The states of the linearized model are the fluid level errors $x(t)=h(t)-h^{d}$ which are assumed to be measurable for a state-feedback control, \emph{i.e.}, $C=I$ and $D=0$. The matrix $A$ is assumed to be completely unknown, which includes the parameters related to the pipes and tanks geometries, discharge coefficients, etc. An application of Algorithm~\ref{alg} to this problem with
    $$
    Q=I,\, R=I,\, t_f=1
    $$
    provides the optimal control input as shown in Figure~\ref{fig:tank}. In this experiment, the initial control input is chosen as $u_0=0$ and the step size is chosen as $\alpha=1$.

\begin{figure}
     \centering
     
     \begin{subfigure}[b]{0.49\textwidth}
         \centering
        \includegraphics[width=1\linewidth]{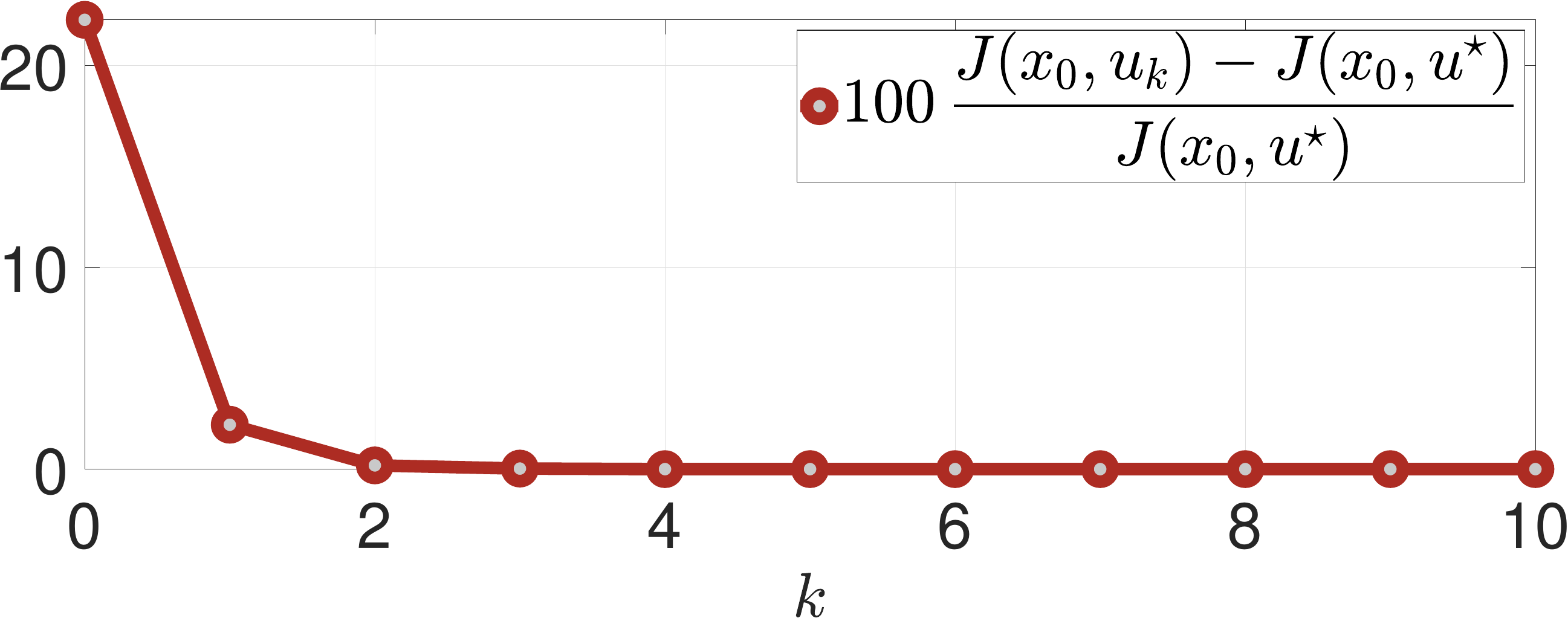}
	    \caption{The relative optimality gap between the optimal solution $u^{\star}$ and the solution $u_k$ obtained from Algorithm~\ref{alg}.}
         \label{fig:tank_J}
     \end{subfigure}
     \hfill
     \begin{subfigure}[b]{0.49\textwidth}
        \centering
        \includegraphics[width=1\linewidth]{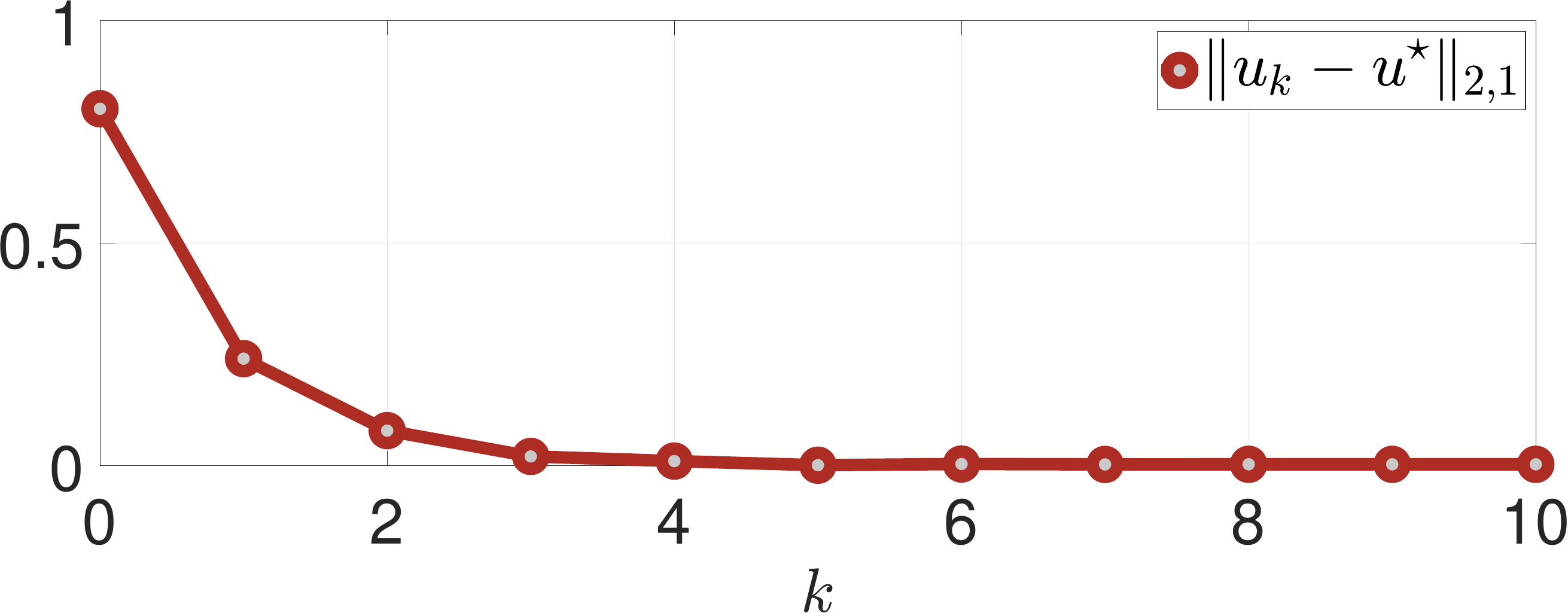}
	 \caption{Comparing the solution $u_k$ obtained from the $k$th iteration in Algorithm~\ref{alg} to the optimal solution $u^{\star}$.}   
         \label{fig:tank_u}
     \end{subfigure}
     
    \caption{Application of Algorithm~\ref{alg} in Example~\ref{ex:tank}.}
    \label{fig:tank}
\end{figure}

\subsection{Example 2}\label{ex:motor}
Consider a rotational electromechanical system described by the state-space equations~\cite[\S2.7]{ohiobook}
\begin{align}\label{eqn:motor}
    \frac{d}{dt}{x}(t)&=\begin{bmatrix}0&1\\-2&-3\end{bmatrix}x(t)+
    \begin{bmatrix}0\\2\end{bmatrix}u(t)\nonumber\\
    y(t)&=\begin{bmatrix}1&0\end{bmatrix}x(t)+\eta(t),
\end{align}
where $y$ denotes the motor shaft angular velocity, $u$ is the applied voltage and $\eta$ is the measurement noise. Assuming all the state-space matrices in (\ref{eqn:motor}) are unknown, desired is the optimal state-feedback controller (\ref{eqn:u=-Kx_IH}) that minimizes the infinite-horizon cost
$$
J(x_0,u)=\frac{1}{2}\int_0^{\infty}y^2(t)+2u^2(t)dt,
$$
where $x_0=\begin{bmatrix}1&1\end{bmatrix}'$. For this purpose, we use Algorithm~\ref{alg} with $k=11$ iterations and the horizon length $t_f=4$. The result is shown in Figures~\ref{fig:infH_u} and \ref{fig:infH_K}, where it is assumed there is no measurement noise ($\eta(t)=\eta_x(t)=0$). Next, we assume both the output and the state measurements are contaminated by zero-mean noises as discussed in Section~\ref{sec:infH}. In this case, using Algorithm~\ref{alg} with the same number of iterations and horizon length ($k=11$, $t_f=4$) yields in the input signal $\tilde{u}_{11}$ (\ref{eqn:estima}), which is used to obtain the noisy data matrices in (\ref{eqn:Xtild,Utild}). These matrices are then used to solve (\ref{eqn:Xtild.Ktild.Utild}) for $\tilde{K}_{11,4}$. We repeat this process for $\overline{\omega}=400$ trials, each realized with new instances of the output and state measurement noises $\eta$ and $\eta_x$ drawn from the same distribution. For comparison purposes, we have solved the same problem with the algorithm provided in \cite{auto2012}. The estimated gains obtained from both approaches in different trials are plotted in Figure~(\ref{fig:noiseballs}). To recover the optimal gain from its noisy estimates, one can find the averaged solution $\overline{K}_{11,4,\overline{\omega}}$ (see Figure~\ref{fig:trials}). However, as it is observed in Figure~(\ref{fig:trials}), unlike Algorithm~\ref{alg}, the bias introduced by the algorithm in \cite{auto2012} does not vanish by averaging the solutions. Further, Algorithm~\ref{alg} took $0.0451 \textnormal{s}$ on average to complete one trial compared with $1.1766 \textnormal{s}$ elapsed in the algorithm provided in \cite{auto2012} in this experiment. 


\begin{figure}
     \centering
     
     \begin{subfigure}[b]{0.49\textwidth}
         \centering
        \includegraphics[width=1\linewidth]{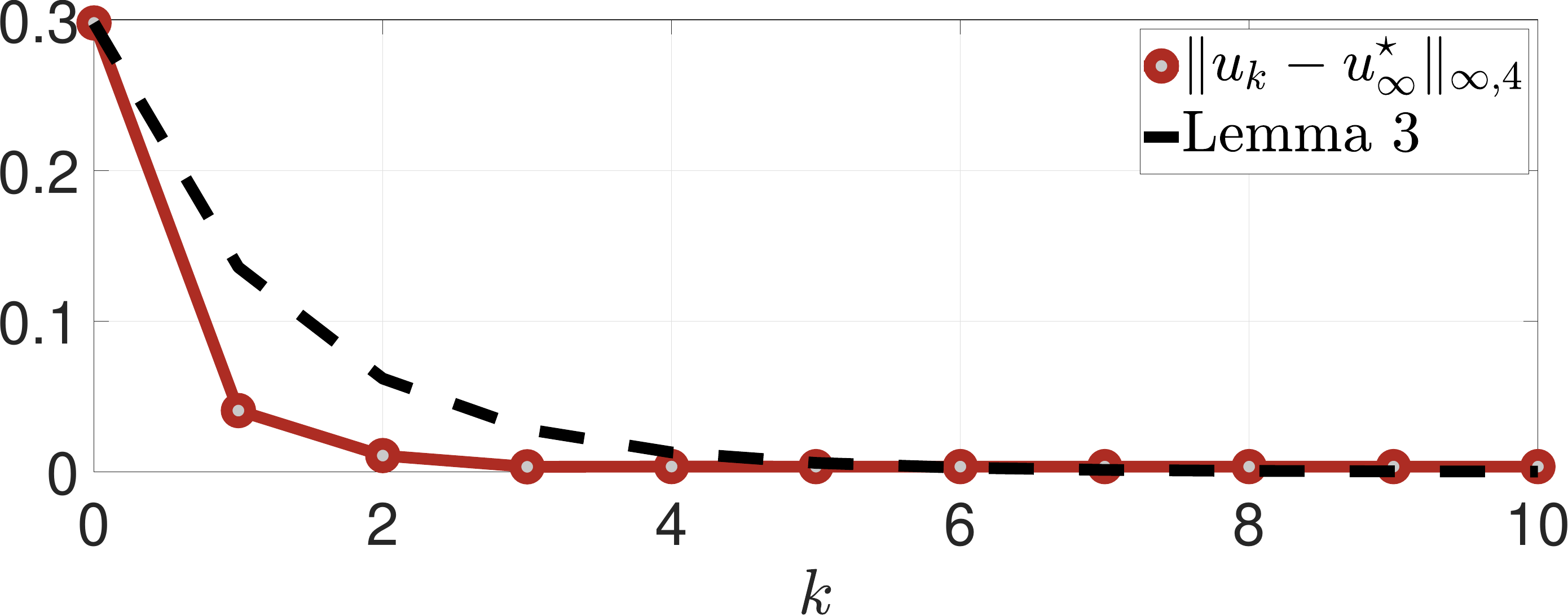}
	    \caption{Comparing the solution $u_k$ obtained from the $k$th iteration in Algorithm~\ref{alg} to the optimal solution $u^{\star}$. The theoretical upper bound guaranteed by Lemma~\ref{lem:convergence_in_Linf} is also shown in this figure.}
         \label{fig:infH_u}
     \end{subfigure}
     \vspace{1mm}
     \hfill
     \begin{subfigure}[b]{0.49\textwidth}
        \centering
        \includegraphics[width=1\linewidth]{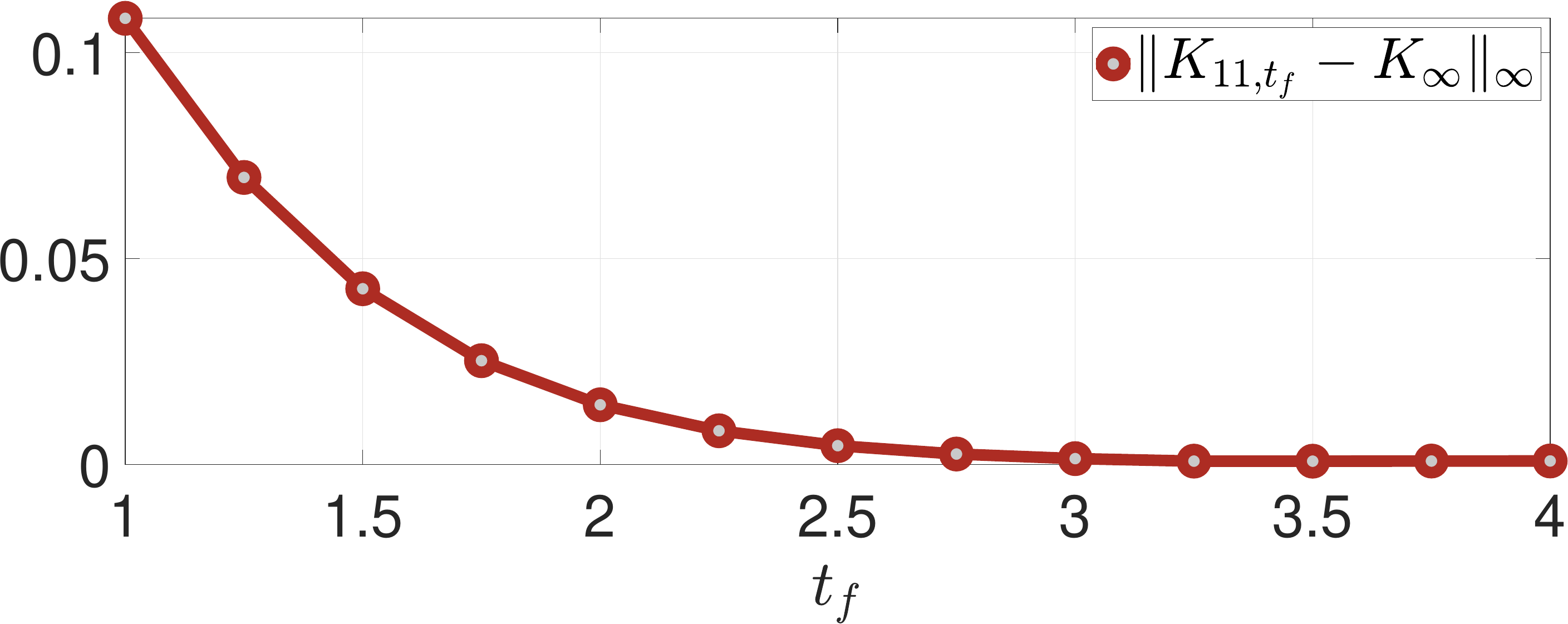}
	    \caption{Comparing the solution $K_{11,t_f}$ of (\ref{eqn:K_k,tf}) to the optimal LQR gain $K_{\infty}$ (\ref{eqn:lqrGain_IH}). The solution obtained from Algorithm~\ref{alg} converges to the optimal solution as $t_f$ is increased.}
         \label{fig:infH_K}
     \end{subfigure}
     \vspace{1mm}
     \vfill
          \begin{subfigure}[b]{0.49\textwidth}
         \centering
        \includegraphics[width=1\linewidth]{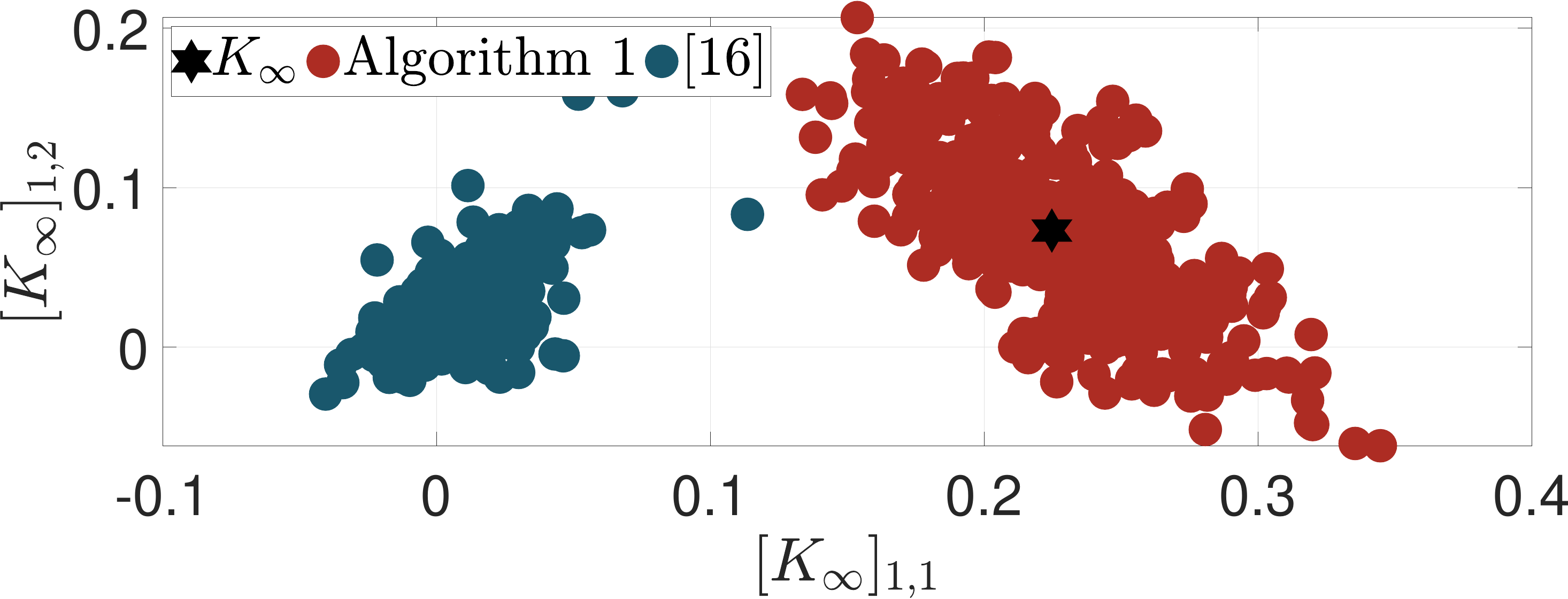}
	    \caption{The optimal LQR gain $K_{\infty}$ estimated by using the algorithms provided in this paper and \cite{auto2012} when there is measurement noise.}
         \label{fig:noiseballs}
     \end{subfigure}
     \hfill
     \begin{subfigure}[b]{0.49\textwidth}
        \centering
        \includegraphics[width=1\linewidth]{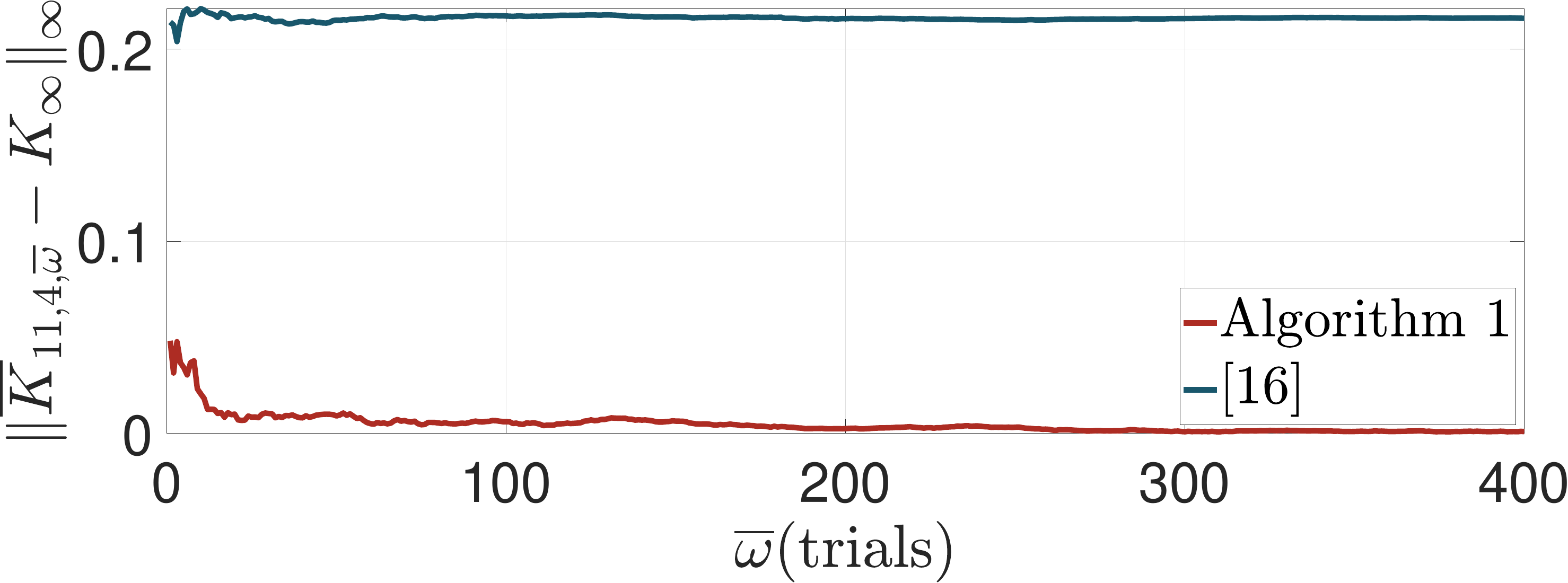}
	    \caption{The averaged solution from Algorithm~\ref{alg} converges to the optimal LQR gain $K_{\infty}$ as more trials are considered.}
         \label{fig:trials}
     \end{subfigure}
     
    \caption{Application of Algorithm~\ref{alg} in Example~\ref{ex:motor}.}
    \label{fig:J,Ktanks}
\end{figure}

\section{Conclusion}
We have shown that the two-point boundary value problems that arise from the application of the minimum principle to linear-quadratic problems can be implemented in a model-free way when the system is externally symmetric. Based on this result, we have provided an algorithm with guaranteed convergence towards the optimal solution of finite-horizon continuous-time output-regulation problems for externally symmetric systems. It was shown that when the system norm is small, the same algorithm can also be used to find the optimal state-feedback gains for infinite-horizon problems. This algorithm operates in a completely model-free fashion and has favorable properties including low computational complexity, simple data sampling strategy, and no estimation bias when there is measurement noise.
\newpage
\appendix
\section{Appendix}

\begin{lemma}\label{lem:cov}
    Let Assumption~\ref{ass:noise_FH} hold. If $\mathcal{K}$ is a bounded linear operator in $\mathcal{L}_2^m(0,t_f)$, then
    $
    \int_0^{t_f}\Vert {\rm Cov}(\mathcal{K}\eta;t)\Vert_2 dt\leq
    m\Vert\mathcal{K}\Vert^2_{2,t_f}
    \sigma^2 t_f.
    $
\end{lemma}
\begin{proof}
    Since ${\rm Cov}(\mathcal{K}\eta;t)\succeq 0$, one can write
    \begin{align*}
        \int_0^{t_f} \Vert{\rm Cov}(\mathcal{K}\eta;t)\Vert_2 dt &\leq
        \int_0^{t_f}{\rm tr}({\rm Cov}\bigl(\mathcal{K}\eta;t) \bigr) dt \\
        &=\int_0^{t_f}\mathbb{E}
        \Bigl({\rm tr}
        \bigl(\mathcal{K}\eta(t){\mathcal{K}\eta}'(t)\bigr)\Bigr)dt \\
        &=\int_0^{t_f}\mathbb{E}(\Vert\mathcal{K}\eta(t)\Vert_{2}^2)dt\\
        &=\mathbb{E}(\Vert\mathcal{K}\eta\Vert_{2,t_f}^2)\\
        &\leq \Vert\mathcal{K}\Vert_{2,t_f}^2
        \mathbb{E}(\Vert \eta\Vert_{2,t_f}^2)\\
        &=\Vert\mathcal{K}\Vert_{2,t_f}^2
        \mathbb{E}\left(\int_0^{t_f}{\rm tr}\bigl(\eta(t)\eta'(t)\bigr) dt\right)\\
        &=\Vert\mathcal{K}\Vert_{2,t_f}^2
        \int_0^{t_f}{\rm tr}\bigl( {\rm Cov}(\eta;t)\bigr)dt\\
        &\leq m\Vert\mathcal{K}\Vert^2_{2,t_f}
    \int_0^{t_f}\Vert {\rm Cov}(\eta;t)\Vert_2 dt \\
    &\leq
    m\Vert\mathcal{K}\Vert^2_{2,t_f}
    \sigma^2 t_f.
    \end{align*}
\end{proof}
\begin{lemma}\label{lem:cov_inf}
    Let Assumption~\ref{ass:noise_IH} hold. If $\mathcal{K}$ is a bounded linear operator in $\mathcal{L}_{\infty}^m(0,t_f)$, then
    $
    \Vert{\rm Cov}(\mathcal{K}\eta;t)\Vert_{\infty} \leq m^{3/2} e^2\Vert \mathcal{K}\Vert_{\infty,t_f}^2
    $ for all $t\in[0,t_f]$.
\end{lemma}
\begin{proof}
    \begin{align*}
        \Vert{\rm Cov}(\mathcal{K}\eta;t)\Vert_{\infty} &\leq m^{1/2} \Vert{\rm Cov}(\mathcal{K}\eta;t)\Vert_{2}\\
        &\leq m^{1/2} {\rm tr}\bigl({\rm Cov}(\mathcal{K}\eta;t)\bigr)\\
        &=m^{1/2} \mathbb{E}\Bigl({\rm tr}\bigl(\mathcal{K}\eta(t){\mathcal{K}\eta}'(t)\bigr)\Bigr)\\
        &=m^{1/2} \mathbb{E}\left(\Vert\mathcal{K}\eta(t)\Vert^2_2\right)\\
        &\leq m^{3/2} \mathbb{E}\left(\Vert\mathcal{K}\eta(t)\Vert^2_{\infty}\right)\\
        &\leq m^{3/2} \mathbb{E}\left(\Vert\mathcal{K}\eta\Vert^2_{\infty,t_f}\right)\\
        &\leq m^{3/2} \Vert \mathcal{K}\Vert_{\infty,t_f}^2 \mathbb{E}\left(\Vert \eta\Vert^2_{\infty,t_f}\right)\\
        &\leq m^{3/2} e^2 \Vert \mathcal{K}\Vert_{\infty,t_f}^2.
    \end{align*}
\end{proof}

\subsection{Proof of Theorem~\ref{thm:noiseFH}} \label{sec:proof_thm:noiseFH}

    In the presence of measurement noise, the update step of Algorithm~\ref{alg} is written as
    \begin{align}\label{eqn:update_utild}
    \tilde{u}_{k+1}&=(1-\alpha)\tilde{u}_k+\alpha\tilde{\mathcal{T}}_k(\tilde{u}_k)\nonumber\\
    &=(1-\alpha)\tilde{u}_k+\alpha\bigl(\mathcal{T}(\tilde{u}_k)+\zeta_k\bigr)\nonumber\\
    &=(1-\alpha)\tilde{u}_k+\alpha\bigl(\mathcal{T}(u_k)+\mathcal{S}(v_k)+\zeta_k\bigr),
    \end{align}
    where $\tilde{u}_k$ is the input signal in the $k$th iteration. It is deduced from (\ref{eqn:update_utild}) that the propagated noise $v_k$ in (\ref{eqn:estima}) satisfies the recurrence    
    \begin{align}
        \nu_{k+1}&=(1-\alpha)u_k+\alpha\mathcal{T}(u_k)-u_{k+1}+(1-\alpha)\nu_{k}\nonumber\\
        &+\alpha\mathcal{S}(\nu_k)+\alpha\zeta_k \nonumber\\
        &=\bigl((1-\alpha)\mathcal{I}+\alpha\mathcal{S}\bigr)\nu_k+\alpha\zeta_k\label{PRE(1-a)I+aS+azeta}\\
        &=R^{-1/2} \mathcal{M}_\alpha R^{1/2}\nu_k+\alpha\zeta_k,\label{(1-a)I+aS+azeta}
    \end{align}
    where $\mathcal{M}_\alpha$ is given by (\ref{eqn:M}). As $\nu_0=0$, it follows from (\ref{(1-a)I+aS+azeta}) and (\ref{zeta}) that
    \begin{align}\label{usingintproof}
        \nu_k&=\alpha\sum_{i=0}^{k-1} R^{-1/2}\mathcal{M}_{\alpha}^i R^{1/2} \zeta_{k-1-i} \nonumber\\
        &=\sum_{i=0}^{k-1} \mathcal{K}_{1,i}\eta_{1,k-1-i}+ \mathcal{K}_{2,i}\eta_{2,k-1-i}, \quad k\in\mathbb{N}
    \end{align}
    where
    \begin{align}\label{K1K2}
        \mathcal{K}_{1,i}&=-\alpha R^{-1/2}\mathcal{M}_{\alpha}^i R^{-1/2}\Sigma_e \mathcal{J}\mathcal{G}\Sigma_e Q \mathcal{J}, \nonumber\\
        \mathcal{K}_{2,i}&=-\alpha R^{-1/2}\mathcal{M}_{\alpha}^i R^{-1/2}\Sigma_e \mathcal{J}
    \end{align}
    are bounded linear operators in $\mathcal{L}^m_2(0,t_f)$.
    Therefore,
    $$
    \mathbb{E}( \nu_k)=
    \sum_{i=0}^{k-1} \mathcal{K}_{1,i}\mathbb{E}(\eta_{1,k-1-i})+ \mathcal{K}_{2,i}\mathbb{E}(\eta_{2,k-1-i})=0.
    $$
    Next we show that the covariance ${\rm Cov}(\nu_k;t)=\mathbb{E}\bigl(\nu_k(t)\nu_k'(t)\bigr)$ is bounded. Since the measurement noises
    $$
    \eta_{j,i}, \quad j=1,2,\,i=0,1,\dots,k-1
    $$
    are zero-mean pairwise-independent random variables, one can write from (\ref{usingintproof}) that
    \begin{align}\label{corvet}
        &{\rm Cov}(\nu_k;t)=\\
        &\sum_{i=0}^{k-1}
        \bigl(
        {\rm Cov}\left(\mathcal{K}_{1,i}\eta_{1,k-1-i};t\right)+ {\rm Cov}\left(\mathcal{K}_{2,i}\eta_{2,k-1-i};t\right)
        \bigr).\nonumber
    \end{align}
    By using Lemma~\ref{lem:cov} and noting the fact that $\Vert \mathcal{M}_{\alpha}\Vert_{2,t_f}<1$ holds from Theorem~\ref{thm}, it is deduced from (\ref{corvet}) that
    \begin{align*}
        &\int_0^{t_f}\Vert{\rm Cov}(\nu_k;t)\Vert_2 dt\\
        &\leq m \sigma^2 t_f \sum_{i=0}^{k-1} \left(
        \Vert \mathcal{K}_{1,i}\Vert^2_{2,t_f} +
        \Vert \mathcal{K}_{2,i}\Vert^2_{2,t_f}\right)\\
        &\leq m \sigma^2 \alpha^2 t_f\Vert R^{-1}\Vert_2^2
        \left( \Vert\mathcal{G}\Vert^2_{2,t_f}\Vert Q \Vert^2_2+1 \right)
        \bigl(1-\Vert\mathcal{M}_\alpha\Vert_{2,t_f}^{2}\bigr)^{-1}\\
        &<\infty
    \end{align*}
    holds for all $k\in\mathbb{N}$.

\hfill$\square$

\subsection{Proof of Theorem~\ref{thm:noiseini}} \label{sec:proof_thm:noiseini}
By following a similar procedure as in Proof of Theorem~\ref{thm:noiseFH}, we obtain the following recurrence relation for the propagated noise $\nu$
    \begin{align}
        \nu_{k+1}&=R^{-1/2} \mathcal{M}_\alpha R^{1/2}\nu_k+\alpha r_{\rho_k},\label{(1-a)I+aS+azeta_ini}
    \end{align}
    where $\mathcal{M}_\alpha$ is given by (\ref{eqn:M}). Since the algorithm is initiated deterministically, we have $\nu_0=0$. Thus, it follows from (\ref{(1-a)I+aS+azeta_ini}) that
    \begin{align}\label{usingintproof_ini}
        \nu_k&=\alpha\sum_{i=0}^{k-1} R^{-1/2}\mathcal{M}_{\alpha}^i R^{1/2} r_{\rho_{k-1-i}}\nonumber\\
        &=\sum_{i=0}^{k-1} \mathcal{K}_{3,i} d_{\rho_{k-1-i}}
        , \quad k\in\mathbb{N}
    \end{align}
    where
    $$
    \mathcal{K}_{3,i}=-\alpha R^{-1/2}\mathcal{M}_{\alpha}^iR^{-1/2}\Sigma_e\mathcal{J}\mathcal{G}\Sigma_eQ\mathcal{J}
    $$
is a bounded linear operators in $\mathcal{L}^m_2(0,t_f)$. Hence, we obtain
$$
       \mathbb{E}(  \nu_k)=\sum_{i=0}^{k-1} \mathbb{E}\bigl(\mathcal{K}_{3,i} d_{\rho_{k-1-i}}\bigr)
       =\sum_{i=0}^{k-1} \mathcal{K}_{3,i} \mathbb{E}\bigl(d_{\rho_{k-1-i}}\bigr),
$$
where $\mathbb{E}\bigl(d_{\rho_{k-1-i}}(t)\bigr)=C\exp(At)\mathbb{E}(\rho_{k-1-i})=0$, which proves (\ref{E=0FHini}). Next, we show that the covariance ${\rm Cov}(\nu_k;t)=\mathbb{E}\bigl(\nu_k(t)\nu_k'(t)\bigr)$ is bounded. Since $\rho_{k-1-i}$ are zero-mean and independent, it is deduced from (\ref{usingintproof_ini}) that
    \begin{align}\label{corvetini}
        {\rm Cov}(\nu_k;t)=
        \sum_{i=0}^{k-1}
        {\rm Cov}\left(
        \mathcal{K}_{3,i} d_{\rho_{k-1-i}};t\right).
    \end{align}
The natural response of the system $d_{\rho_{k-1-i}}$ is bounded within the range $t\in[0,t_f]$. Therefore, we have
\begin{align*}
    \bigl\Vert {\rm Cov}(d_{\rho_k};t)\bigr\Vert_2&=\Vert \mathbb{E}\bigl( d_{\rho_k}(t)d'_{\rho_k}(t)\bigr)\Vert_2\\
    &\leq \textnormal{tr}\bigl(\mathbb{E}( d_{\rho_k}(t)d'_{\rho_k}(t))\bigr)\\
    &\leq \mathbb{E}\bigl(\textnormal{tr}( d_{\rho_k}(t)d'_{\rho_k}(t))\bigr)\\
    &= \mathbb{E}\bigl( \Vert d_{\rho_k}(t)\Vert_2^2 )\bigr)\\
    &\leq \sigma^2<+\infty, \quad t\in[0,t_f].
\end{align*}
Hence, Assumption~\ref{ass:noise_FH} is satisfied for the stochastic process $d_{\rho_k}$ for all $k$. Therefore, one may invoke Lemma~\ref{lem:cov} in (\ref{corvetini}), to obtain
    \begin{align*}
        \int_0^{t_f}\Vert{\rm Cov}(\nu_k;t)\Vert_2dt&\leq
        \sum_{i=0}^{k-1}
        \int_0^{t_f}\bigl\Vert
        {\rm Cov}\left(
        \mathcal{K}_{3,i} d_{\rho_{k-1-i}};t\right)\bigr\Vert_2 dt\nonumber\\
        &\leq \sum_{i=0}^{k-1}
        m\sigma^2 t_f \Vert \mathcal{K}_{3,i}\Vert^2_{2,t_f}\nonumber\\
        &\leq
        m\sigma^2 t_f
\alpha^2
\Vert R^{-1/2}\Vert^4_{2}
\Vert \mathcal{G}\Vert^2_{2,t_f}
\Vert Q \Vert^2_{2}\\
&\times
\sum_{i=0}^{k-1}
\Vert \mathcal{M}_{\alpha}^i\Vert^2_{2,t_f}\\
&\leq
\frac{m\sigma^2 t_f
\alpha^2
\Vert R^{-1/2}\Vert^4_{2}
\Vert \mathcal{G}\Vert^2_{2,t_f}
\Vert Q \Vert^2_{2}}{1-\Vert \mathcal{M}_{\alpha}\Vert^2_{2,t_f}}\\
&\leq \infty.
    \end{align*}

\hfill$\square$

\subsection{Proof of Theorem~\ref{thm:noise_inf}} \label{sec:proof_thm:noise_inf}
The proof of identity (\ref{eqn:E(nu)=0_inf}) is similar to the proof of (\ref{E=0FH}) in Theorem~\ref{thm:noiseFH}. To prove the noise $\nu_k$ has a bounded variance, we plug-in (\ref{eqn:M}) in the definitions (\ref{K1K2}) and use the triangular inequality in (\ref{corvet}) to obtain
\begin{align*}
    &\Vert{\rm Cov}(\nu_{k};t)\Vert_{\infty}\leq \\
    &\sum_{i=0}^{k-1}
    \left\Vert{\rm Cov}\left(\alpha\bigl((1-\alpha)\mathcal{I}+\alpha\mathcal{S}\bigr)^{i}R^{-1}\Sigma_e \mathcal{J}\eta_{2,k-1-i};t\right)\right\Vert_{\infty}+\\
    &\bigl\Vert{\rm Cov}\bigl(\alpha\bigl((1-\alpha)\mathcal{I}+\alpha\mathcal{S}\bigr)^{i}R^{-1}\Sigma_e \mathcal{J}\mathcal{G}\Sigma_e Q \mathcal{J} \eta_{1,k-1-i};t\bigr)\bigr\Vert_{\infty},\\
\end{align*}
where using Lemma~\ref{lem:cov_inf} gives
\begin{align}\label{eqn:last}
    &\Vert{\rm Cov}(\nu_{k};t)\Vert_{\infty}\leq \nonumber\\
    &m^{3/2} e^2\sum_{i=0}^{k-1}
    \left(\left\Vert
    \alpha\bigl((1-\alpha)\mathcal{I}+\alpha\mathcal{S}\bigr)^{i}R^{-1}\Sigma_e \mathcal{J}\mathcal{G}\Sigma_e Q \mathcal{J} \right\Vert_{\infty,t_f}^2
    \right.\nonumber\\
    &+\left.\left\Vert
    \alpha\bigl((1-\alpha)\mathcal{I}+\alpha\mathcal{S}\bigr)^{i}R^{-1}\Sigma_e \mathcal{J}\right\Vert_{\infty,t_f}^2\right)\leq\nonumber\\
    & m^{3/2} e^2 \alpha^2
    \Vert R^{-1}\Vert_{\infty}^2\left( \Vert \mathcal{G}\Vert_{\infty,t_f}^2 \Vert Q\Vert_{\infty}^2+1\right)\times\nonumber\\
    &\sum_{i=0}^{k-1}
    \Vert (1-\alpha)\mathcal{I}+\alpha\mathcal{S}\Vert_{\infty,t_f}^{2i}\leq
    \nonumber\\
    & m^{3/2} e^2 \alpha^2 \Vert R^{-1}\Vert_{\infty}^2\left( \Vert \mathcal{G}\Vert_{\infty,t_f}^2 \Vert Q\Vert_{\infty}^2+1\right)\times\nonumber\\
    &\sum_{i=0}^{k-1}
    \left(1-\alpha+\alpha \Vert R^{-1} \Vert_{\infty} \Vert {G}\Vert_{\rm pk}^2 \Vert Q\Vert_{\infty}\right)^{2i},
\end{align}
in which the last inequality follows from (\ref{eqn:infproof1}). Condition (\ref{eqn:conv_cond_inf}) ensures that the series in (\ref{eqn:last}) is convergent. Therefore, we obtain
\begin{align*}\label{eqn:lastlast}
    \Vert{\rm Cov}(\nu_{k};t)\Vert_{\infty}&\leq
    \frac{
    m^{3/2} e^2 \alpha^2 \Vert R^{-1}\Vert_{\infty}^2\left( \Vert G\Vert_{\rm pk}^2 \Vert Q\Vert_{\infty}^2+1\right)
    }{
    1-\left(1-\alpha+\alpha \Vert R^{-1} \Vert_{\infty} \Vert G\Vert_{\rm pk}^2 \Vert Q\Vert_{\infty}\right)^{2}}
\end{align*}
using (\ref{Gtf<Gpk}). The above variance bound is independent of the horizon length and hence, it is bounded, as $t_f\to\infty$.

\hfill$\square$

\printbibliography
\end{document}